\documentclass[a4j]{article}
\usepackage{lipsum}
\usepackage{amsmath,amsthm,amssymb}
\usepackage{graphicx}
\usepackage{hhline,url}
\usepackage[round]{natbib}
\usepackage{mathrsfs}
\usepackage{authblk}

\usepackage{algorithm,algpseudocode}

\algnewcommand{\Initialize}[1]{%
	\State \textbf{Initialize:}
	\State \hspace*{\algorithmicindent}\parbox[t]{0.8\linewidth}{\raggedright #1}
}

\newtheorem{thm}{Theorem}[]
\newtheorem*{thm*}{Theorem}
\newtheorem{lem}[thm]{Lemma}
\newtheorem{remark}{Remark}[]
\newenvironment{rem}{\begin{remark}\rm}{\end{remark}}
\newtheorem{prop}[thm]{Proposition}
\newtheorem{Definition}[thm]{Definition}

\newtheorem{Corollary}[thm]{Corollary}
\newenvironment{cor}{\begin{Corollary}}{\end{Corollary}}
\newtheorem{Example}[thm]{Example}
\newenvironment{eg}{\begin{Example}\rm}{\end{Example}}
\newtheorem{algor}[thm]{Method}

\newtheorem{Condition}[thm]{Condition}

\newcommand{\R}{\mathbb{R}}
\newcommand{\dd}{\,\mathrm{d}}
\newcommand{\ve}{\varepsilon}

\renewcommand{\phi}{\varphi}

\newcommand{\X}{\mathcal{X}}
\newcommand{\bm}[1]{{\mbox{\boldmath $#1$}}}

\DeclareMathOperator{\cv}{conv}

\DeclareMathOperator{\supp}{supp}

\newcommand{\lmid}{\,\middle|\,}

\newcommand{\E}[1]{\mathbb{E}\!\left[#1\right]}
\renewcommand{\P}[1]{\mathbb{P}\!\left(#1\right)}
\newcommand{\ord}[1]{\mathcal{O}\!\left(#1\right)}

\newcommand{\ip}[1]{\left\langle #1 \right\rangle}
\newcommand{\nip}[1]{\langle #1 \rangle}

\newcommand{\vertiii}[1]{{\left\vert\kern-0.25ex\left\vert\kern-0.25ex\left\vert #1 
		\right\vert\kern-0.25ex\right\vert\kern-0.25ex\right\vert}}

\usepackage{bbm}
\usepackage{xcolor}

\DeclareMathOperator{\dist}{dist}
\title{Estimating the probability that a given vector is in the convex hull of a random sample}
\author{Satoshi Hayakawa\footnote{hayakawa@maths.ox.ac.uk}}
\author{Terry Lyons\footnote{Supported by
the DataSıg Program [EP/S026347/1] and the Alan Turing Institute [EP/N510129/1].}}
\author{Harald Oberhauser$^\dagger$}
\affil{Mathematical Institute, University of Oxford}
\date{}

\begin{document}
\maketitle

\begin{abstract}
    For a $d$-dimensional random vector $X$,
    let $p_{n, X}(\theta)$ be the probability that
    the convex hull of $n$ independent copies of $X$ contains a given point $\theta$.
    We provide several sharp inequalities regarding $p_{n, X}(\theta)$
    and $N_X(\theta)$ denoting the smallest $n$ for which $p_{n, X}(\theta)\ge1/2$.
    As a main result, we derive the totally general inequality
    $1/2 \le \alpha_X(\theta)N_X(\theta)\le 3d + 1$,
    where $\alpha_X(\theta)$ (a.k.a. the Tukey depth)
    is the minimum probability that $X$ is in a fixed closed halfspace
    containing the point $\theta$.
    We also show several applications of our general results:
    one is a moment-based bound on $N_X(\E{X})$,
    which is an important quantity in randomized approaches to cubature construction
    or measure reduction problem.
    Another application is the determination of the canonical convex body
    included in a random convex polytope given by
    independent copies of $X$,
    where our combinatorial approach allows us to generalize
    existing results in random matrix community significantly.
\end{abstract}



\section{Introduction}
Consider generating independent and identically distributed
$d$-dimensional random vectors.
How many vectors do we have to generate in order that a point $\theta\in\R^d$
is contained in the convex hull of the sample
with probability at least $1/2$?
More generally, what is the probability of the event with
an $n$-point sample for each $n$?
These questions were first solved for a general distribution which has a
certain symmetry about $\theta$ by \citet{wen62}.
Let us describe the problem more formally.

Let $X$ be a $d$-dimensional random vector
and $X_1,X_2,\ldots$ be independent copies of $X$.
For each $\theta\in\R^d$ and positive integer $n$,
define
\[
    p_{n, X}(\theta):=\P{\theta\in\cv\{X_1,\ldots,X_n\}},
\]
where
$\cv A:=\{\sum_{i=1}^m\lambda_ix_i\mid m\ge1,\ x_i\in A,\ \lambda_i\ge 0,\ \sum_{i=1}^m\lambda_i = 1\}$
denotes the convex hull of a set $A\subset\R^d$.
We also define
\[
    N_X(\theta):=\inf\{n\mid p_{n,X}(\theta)\ge 1/2\}
\]
as the reasonable number of observations we need.
As $p_{n, X}$ and $N_X$ are only dependent on the probability distribution of $X$,
we also write $p_{n, \mu}$ and $N_\mu$ when $X$ follows the distribution $\mu$.
We want to evaluate $p_{n,X}$ as well as $N_X$ for a general $X$.

\citet{wen62} showed that
\begin{equation}
    p_{n, X}(0)
    = 1 - \frac1{2^{n-1}}\sum_{i=0}^{d-1}\binom{n-1}{i}
    \label{weq}
\end{equation}
holds for an $X$ such that $X$ and $-X$ have the same distribution
and $X_1,\ldots,X_d$ are almost surely linearly independent.
In particular, $N_X(0) = 2d$ holds for such random vectors.
For an $X$ with an absolutely continuous distribution
with respect to the Lebesgue measure,
\citet{wag01} showed more generally that
the right-hand side of \eqref{weq}
is indeed an upper bound of $p_{n, X}$,
and they also characterized the condition for equality
(see Theorem \ref{thm-wag}).
Moreover, \citet{kab20} recently
gave an explicit formula for $p_{n, X}$ when $X$ is a shifted Gaussian.

In this paper, our aim is to give generic bounds of $p_{n, X}$ and $N_X$,
and we are particularly interested in the upper bound of $N_X$,
which is opposite to the bound given by \citet{wag01}.
Estimating $p_{n, X}$ and $N_X$ is of great interest from application,
which ranges from numerical analysis to statistics, and compressed sensing.
As a by-product, we also give a general result explaining
the deterministic body included in the random polytope
$\cv\{X_1, \ldots, X_n\}$,
which is a sharp generalization of a recent work in the random matrix community
\citep{gue19}.
The remainder of this section will explain
more detailed motivation from related fields
and implications of our results.

Throughout the paper,
let $\ip{\cdot, \cdot}$ be any inner product on $\R^d$,
and $\|\cdot\|$ be the norm it induces.

\subsection{Cubature and measure reduction}\label{cubature}
Let $\mu$ be a Borel probability measure on
some topological space $\X$.
Consider $d$ integrable functions
$f_1, \ldots, f_d : \X \to \R$.
Then, we know the existence of ``good reduction" of $\mu$
by Tchakaloff's theorem \citep{tch57,bay06}:
\begin{thm}[Tchakaloff]
    There are $d+1$ points $x_1, \ldots, x_{d+1}\in \supp\mu$
    and weights $w_1, \ldots, w_{d+1}\ge0$ such that $w_1+\cdots+w_{d+1} = 1$
    and
    \begin{equation}
        \int_\X f_i(x)\dd\mu(x) = \sum_{j=1}^{d+1}w_jf_i(x_j)
        \label{eq:cub}
    \end{equation}
    holds for each $i = 1, \ldots, d$.
\end{thm}
The proof is essentially given by classical Carath\'{e}odory's theorem.
The points and weights treated in Tchakaloff's theorem
is an important object in the field of numerical integration,
called cubature \citep{str71}.
An equivalent problem is also treated as a beneficial way of data compression
in the field of data science \citep{maa19,cos20}.
A typical choice of test function $f_i$ is monomials when $\X$ is a subset
of an Euclidean space,
so the integration with respect to the measure $\sum_{j=1}^{d+1}w_j\delta_{x_j}$
is a good approximation of $\int_\X f\dd\mu$ for a smooth integrand $f$.
However, constructions under general setting are also useful;
for example, in the cubature on Wiener space \citep{lyo04},
$\X$ is the space of continuous paths,
$\mu$ is the Wiener measure,
and the test functions are iterated integrals of paths.

To this generalized cubature construction (or measure reduction)
problem,
there are efficient deterministic approaches \citep{lit12,tch15,maa19}
when $\mu$ is discrete.
Using randomness for construction is recently considered
\citep{cos20,hayakawa-MCCC} and it is important to know $p_{n, X}(\E{X})$
for the $d$-dimensional random variable
\[
    X=\bm{f}(Y) = (f_1(Y),\ldots,f_d(Y))^\top,
\]
where $Y$ is drawn from $\mu$.
Indeed, once we have $\E{X}\in\cv\{X_1, \ldots, X_n\}$
($X_i = \bm{f}(Y_i)$ are independent copies of $X$),
then we can choose $d+1$ points and weights satisfying \eqref{eq:cub}
by solving a simple linear programming problem.
Evaluation of $N_X$ is sought for estimating the computational complexity
of this naive scheme.

\subsection{Statistical depth}\label{sec:dep}
From the statistical context,
$p_{d+1, X}(\theta)$ for a $d$-dimensional $X$
is called the simplicial depth of $\theta\in\R^d$
with respect to the (population) distribution of $X$ \citep{liu90,cas07},
which can be used for mathematically characterizing the intuitive ``depth"
of each point $\theta$ when we are given the distribution of $X$.
For an empirical measure, it corresponds to the number of simplices
(whose vertices are in the data) containing $\theta$.

There are also a various concepts measuring depth,
all called statistical depth \citep{cas07,mos13}.
One of the first such concepts is the halfspace depth proposed by \citet{tuk75}:
\[
    \alpha_X(\theta):=\inf_{c\in\R^d\setminus\{0\}}\P{\ip{c, X - \theta}\le 0},
\]
which can equivalently defined as the minimum measure
of a halfspace containing $\theta$.
\citet{don92} and \citet{rou99} extensively studied general features
of $\alpha_X$.
We call it the Tukey depth throughout the paper.

Our finding is that these two depth notions are indeed deeply related.
We prove the rate of convergence
$p_{n, X}\to 1$ is essentially determined by $\alpha_X$
(Proposition \ref{new:prop-main}),
and we have a beautiful relation $1/2\le \alpha_XN_X\le 3d+1$ in Theorem \ref{N_X}.

\subsection{Inclusion of deterministic convex bodies}\label{new:intro:int}
Although we have seen the background of the $p_{n, X}(\theta)$,
which only describes the probability of a single vector contained in
the random convex polytope,
several aspects of such random polytopes have been studied
\citep{maj10,hug13}.
In particular, people also studied deterministic convex bodies associated with
the distribution of a random vector.
For example, one consequence of well-known Dvoretzky--Milman's Theorem
(see, e.g., \citealt[Chapter 11]{HDP-book})
is that the convex hull of $n$ independent samples from the $d$-dimensional standard normal
distribution is ``approximately" a Euclidean ball of radius $\sim\sqrt{\log n}$ with high probability
for a sufficiently large $n$.

Mainly from the context of random matrices,
there have been several researches on
the interior convex body of $\cv\{X_1, \ldots, X_n\}$
or its ``absolute" version $\cv\{\pm X_1,\ldots, \pm X_n\}$
for various classes of $X$ such as Gaussian, Rademacher or vector with i.i.d. subgaussian entries
\citep{glu89,gia02,lit05,daf09,gue20}.
One result about the Rademacher vector is the following:
\begin{thm}[{\citealt{gia02}}]
    Let $d$ be a sufficiently large positive integer
    and $X_1, X_2, \ldots$ be independent samples from
    the uniform distribution over the set $\{-1, 1\}^d\subset\R^d$.
    Then, there exists an absolute constant $c>0$ such that,
    for each integer $n\ge d(\log d)^2$,
    we have
    \[
        \cv\{\pm X_1, \ldots, \pm X_n\} \supset
        c\left(\sqrt{\log(n / d)}B_2^d \cap B_\infty^d\right)
    \]
    with probability at least $1 - e^{-d}$.
    Here, $B_2^d$ is the Euclidean unit ball in $\R^d$
    and $B_\infty^d = [-1, 1]^d$.
\end{thm}

Although each of those results in literature was based on its specific
assumptions on the distribution of $X$,
\citet{gue19} found a possible way of treating the results in a unified manner
under some technical assumptions on $X$.
They introduced the floating body associated with $X$
\[
     \tilde{K}^\alpha(X):=\{s\in\R^d \mid \P{\ip{s, X} \ge 1} \le \alpha\}
\]
to our context
(the notation here is slightly changed from the original one),
and argued that,
under some assumptions on $X$,
with high probability, $\cv\{X_1, \ldots, X_n\}$ includes
a constant multiple of the polar body of $\tilde{K}^\alpha(X)$
with $\log(1/\alpha)\sim 1 + \log (n/ d)$.
Note that their main object of interest is the absolute convex hull,
but their results can be extended to the ordinary convex hull
(see \citealt[Remark 1.7]{gue19}).

Let us explain more formally.
Firstly, for a set $A\subset\R^d$,
the polar body of $A$ is defined as
\[
    A^\circ:=\{x\in\R^d\mid \ip{a, x}\le 1\ \text{for all}\ a\in A\}.
\]
Secondly, we shall describe the assumptions used in \citet{gue19}.
Let $\vertiii{\cdot}$ be a norm on $\R^d$
and $\gamma, \delta, r, R > 0$ be constants.
Their assumptions are as follows:
\begin{itemize}
    \item $(\gamma, \delta)$ small-ball condition:
        $\P{|\!\ip{t, X}\!|\ge \gamma\vertiii{t}} \ge \delta$ holds
        for all $t\in\R^d$.
    \item $L_r$ condition with constant $R$:
        $\E{|\!\ip{t, X}\!|^r}^{1/r} \le R\vertiii{t}$ holds
        for all $t\in\R^d$.
\end{itemize}
Under these conditions, they proved the following assertion
by using concentration inequalities.
\begin{thm}[{\citealt{gue19}}]\label{new:thm:gue}
    Let $X$ be a $d$-dimensional symmetric random vector that satisfies the
    small-ball condition and $L_r$ condition for
    a norm $\vertiii{\cdot}$ and constants $\gamma, \delta, r, R>0$.
    Let $\beta\in(0, 1)$ and set
    $\alpha = (en/d)^{-\beta}$.
    Then, there exist a constant $c_0 = c_0(\beta, \delta, r, R/\gamma)$
    and an absolute constant $c_1>0$ such that,
    for each integer $n\ge c_0 d$,
    \[
        \cv\{X_1, \ldots, X_n\} \supset \frac12 \bigl(\tilde{K}^\alpha(X)\bigr)^\circ
    \]
    holds with probability at least $1 - 2\exp( - c_1 n^{1-\beta}d^\beta)$,
    where $X_1, X_2, \ldots$ are independent copies of $X$.
\end{thm}

Though computing $\bigl(\tilde{K}^\alpha(X)\bigr)^\circ$ for individual $X$
is not necessarily an easy task,
this gives us a unified understanding of existing results in terms
of the polar of the floating body $\tilde{K}^\alpha(X)$.
However, its use is limited due to the technical assumptions.
In this paper, we show that we can completely remove the assumptions in
Theorem \ref{new:thm:gue} and obtain a similar statement only with explicit constants
(see Proposition \ref{new:equivalence} and Corollary \ref{new:cor-main}, or the next section).

Finally, we add that this interior body of random polytopes or its radius
is recently reported to be essential
in the robustness of sparse recovery \citep{kra18}
and the convergence rate of greedy approximation algorithms 
\citep{mir17,com19} when the data is random.

\subsection{Organization of the paper}
In this paper, our aim is to derive general inequalities for $p_{n, X}$ and $N_X$.
The main part of this paper is Section \ref{sec2} to \ref{new:sec:int}.
The following is a broad description of the contents of each section.
\begin{itemize}
    \item Section \ref{sec2}: General bounds of $p_{n, X}$
    without specific quantitative assumptions
    \item Section \ref{sec:ep}: Bounds of $p_{n, X}$ uniformly determined by $\alpha_X$
    \item Section \ref{sec4}: Bounds of $N_X(\E{X})$ uniformly determined by the moments of $X$
    \item Section \ref{new:sec:int}: Results on deterministic convex bodies included in random polytopes
\end{itemize}

Let us give more detailed explanation about each section.
Section \ref{sec2} provides generalization of the results of \citet{wag01},
and we give generic bounds of $p_{n, X}(\theta)$
under a mild assumption $p_{d, X}(\theta) = 0$,
which is satisfied with absolutely continuous distributions
as well as typical empirical distributions.
Our main result in Section \ref{sec2} is as follows (Theorem \ref{smoothing}):
\begin{thm*}
    Let $X$ be an arbitrary $d$-dimensional random vector
    and $\theta\in\R^d$.
    If $p_{d, X}(\theta)=0$ holds, then, for any $n\ge m\ge d+1$,
    inequalities
    \[
        p_{n, X}(\theta)
        \le 1- \frac1{2^{n-1}}\sum_{i=0}^{d-1}\binom{n-1}{i},\qquad
        \frac1{2^{n-m}}\frac{\binom{n}{d+1}}{\binom{m}{d+1}}p_{m,X}(\theta)
        \le p_{n, X}(\theta)
        \le\frac{\binom{n}{d+1}}{\binom{m}{d+1}}p_{m,X}(\theta)
    \]
    hold.
\end{thm*}

In Section \ref{sec:ep},
we introduce
$p_{n, X}^\ve$ and $\alpha_X^\ve$ for an $\ve\ge0$,
which are ``$\ve$-relaxation" of $p_{n, X}$ and $\alpha_X$
in that $p_{n, X}^0 = p_{n, X}$ and $\alpha_X^0 = \alpha_X$ hold.
For this generalization, we prove that
the convergence of $p^\ve_{n, X} \to 1$ is uniformly evaluated
in terms of $\alpha_X^\ve$ (Proposition \ref{new:prop-main}),
and obtain the following result (Theorem \ref{2^d}):
\begin{thm*}\label{2^d}
    Let $X$ be an arbitrary $d$-dimensional random vector
    and $\theta\in\R^d$.
    Then, for each $\ve\ge0$ and positive integer $n\ge 3d/\alpha_X^\ve(\theta)$,
    we have
    \[
        p_{n, X}^\ve(\theta) > 1 - \frac1{2^d}.
    \]
\end{thm*}
\noindent Although we do not define $\ve$-relaxation version here,
we can see from the case $\ve = 0$ that,
for example, $N_X(\theta) \le \lceil 3d/\alpha_X(\theta) \rceil$ generally holds
(see also Theorem \ref{N_X}).

In Section \ref{sec4},
we derive upper bounds of $N_X$
without relying on $\alpha_X$,
which may also be unfamiliar.
By using the result in the preceding section and
the Berry--Esseen theorem,
we show some upper bounds of $N_X$
in terms of the (normarized) moments of $X$ as follows (Theorem \ref{main-be}):
\begin{thm*}
    Let $X$ be a centered $d$-dimensional random vector with nonsingular covariance matrix $V$.
    Then,
    \[
        N_X\le 17d\left(1 + \frac94\sup_{c\in\R^d,\|c\|_2=1}
        \E{\left\lvert c^\top V^{-1/2} X\right\rvert^3}^2\right)
    \]
    holds.
\end{thm*}
\noindent Here, $\|\cdot\|_2$ denotes the usual Euclidean norm on $\R^d$.
Note that the right-hand side can easily be replaced by
the moment of $\|V^{-1/2}X\|_2$ (see also Corollary \ref{cor-nx}).

Section \ref{new:sec:int} asserts that
$K^\alpha(X):=\{\theta\in\R^d \mid \alpha_X(\theta) \ge \alpha\}$
($\alpha\in(0, 1)$) is a canonical deterministic body
included in the random convex polytope $\cv\{X_1, \ldots, X_n\}$.
We see in Proposition \ref{new:equivalence}
that this body is essentially equivalent to the
$\bigl(\tilde{K}^\alpha(X)\bigr)^\circ$
mentioned in Section \ref{new:intro:int},
and prove the following (Theorem \ref{new:thm-interior}):
\begin{thm*}
    Let $X$ be an arbitrary symmetric $d$-dimensional random vector,
    and let $\alpha, \delta, \ve\in(0, 1)$.
    If a positive integer $n$ satisfies
    \[
        n \ge \frac{2d}\alpha \max\left\{\frac{\log(1/\delta)}d + \log\frac1\ve,\ 6\right\},
    \]
    then we have, with probability at least $1 - \delta$,
    \[
        \cv\{X_1, \ldots, X_n\} \supset (1-\ve)K^\alpha(X),
    \]
    where $X_1, X_2, \ldots$ are independent copies of $X$.
\end{thm*}
\noindent
A consequence of this theorem (Corollary \ref{new:cor-main})
enables us to remove
the technical assumption of Theorem \ref{new:thm:gue}.

Note that all these results give explicit constants with reasonable magnitude,
which is because of our combinatorial approach typically seen in the proof of
Proposition \ref{new:prop:ineq} and Proposition \ref{key}.
After these main sections,
we give some implications of our results on motivational examples
(introduced in Section \ref{cubature}, \ref{sec:dep}) in Section \ref{sec5-1},
and we finally give our conclusion in Section \ref{sec6}.

\section{General bounds of $p_{n,X}$}\label{sec2}
In this section, we denote $p_{n, X}(0)$ by
only $p_{n, X}$.
As we always have $p_{n, X}(\theta) = p_{n, X-\theta}(0)$,
it suffices to treat $p_{n, X}(0)$ unless we consider
properties of $p_{n, X}$ as a function.

Let us start with easier observations.
Proposition \ref{prop0} and Proposition \ref{naive}
are almost dimension-free.
Firstly, as one expects, the following simple assertion holds.
\begin{prop}\label{prop0}
    For an arbitrary $d$-dimensional random vector $X$ with $\E{X} = 0$ and $\P{X\ne0}>0$,
    we have
    \[
        0 < p_{d+1, X} < p_{d+2, X} < \cdots < p_{n, X} < \cdots \to 1.
    \]
    The conclusion still holds if we only assume $p_{n, X}>0$ for some $n$ instead of $\E{X}=0$.
\end{prop}
\begin{proof}
    For the proof of $p_{2d, X} > 0$, see, e.g., \citet{hayakawa-MCCC}.
    From this and Carath\'{e}odory's theorem, we also have $p_{d+1, X} > 0$.
    We clearly have $p_{n+1, X} \ge p_{n, X}$ for each $n\ge d+1$.
    
    The strict inequality also seems trivial,
    but we prove this for completeness.
    Assume $p_{n+1, X} = p_{n, X}$ for some $n$.
    This implies that $0\not\in\cv\{X_i\}_{i=1}^n\Rightarrow0\not\in\cv\{X_i\}_{i=1}^{n+1}$
    holds almost surely.
    By symmetry, for any $J\subset\{1, \ldots, n+2\}$ with $\lvert J\rvert = n + 1$,
    $0\not\in\cv\{X_i\}_{i=1}^{n+1}\Rightarrow0\not\in\cv\{X_i\}_{i\in J}$ holds almost surely.
    Therefore, we have
    $0\not\in\cv\{X_i\}_{i=1}^n\Rightarrow0\not\in\cv\{X_i\}_{i=1}^{n+2}$ with probability one.
    By repeating this argument,
    we obtain
    \[
        0\not\in\cv\{X_1,\ldots, X_n\}
        \Longrightarrow
        0\not\in\cv\{X_1,\ldots, X_{n+d+1}\}
        \Longrightarrow
        0\not\in\cv\{X_{n+1},\ldots, X_{n+d+1}\}
    \]
    with probability one,
    but this is only possible when $\P{0\not\in\cv\{X_1,\ldots, X_n\}} = 0$
    as $p_{d+1, X} > 0$ and the variables
    $X_{n+1},\ldots, X_{n+d+1}$ are independent from the others.
    This is of course impossible from the assumption $\P{X\ne0}>0$
    (there exists a unit vector $c\in\R^d$ such that $\P{\ip{c, X} > 0} > 0$),
    so we finally obtain $p_{n, X} < p_{n+1, X}$.
    
    Proving $p_{n, X}\to1$ is also easy.
    From the independence, we have
    \begin{align*}
        p_{m(d+1), X}
        &= 1 - \P{0\not\in\cv\{X_1, \ldots, X_{m(d+1)}\}}\\
        &\ge 1 - \P{\bigcap_{k=1}^m\{0\not\in\cv\{X_{(k-1)(d+1)+1}, \ldots, X_{k(d+1)}\}\}} \\
        &=1 - (1-p_{d+1, X})^m \to 1 \qquad (m\to\infty).
    \end{align*}
    This leads to the conclusion combined with the monotonicity of $p_{n, X}$.
    
    Note that we have used the condition $\E{X}=0$ only to ensure $p_{d+1}>0$.
    Hence the latter statement readily holds from the same argument.
\end{proof}

The next one includes a little quantitative relation
among $p_{n, X}$ and $N_X$.
\begin{prop}\label{naive}
    For an arbitrary $d$-dimensional random vector $X$
    and integers $n \ge m \ge d+1$,
    \[
        p_{n, X} \le \binom{n}{m}p_{m, X},\qquad
        N_X \le \frac{n}{p_{n, X}}
    \]
    hold.
\end{prop}

\begin{proof}
    Let $M$ be the number of $m$-point subsets of $\{X_1,\ldots,X_n\}$
    whose convex hull contains $0$.
    Then, we have
    \[
        \E{M}=\sum_{\substack{J\subset\{1,\ldots,n\}\\\lvert J \rvert = m}}
        \P{0\in\cv\{X_i\}_{i\in J}}
        =\binom{n}{m}p_{m, X}.
    \]
    As $p_{n, X}=\P{M\ge 1}\le \E{M}$, we obtain the first inequality.
    
    For the second part,
    we carry out the following rough estimate:
    For the minimum integer $k$ satisfying $(1-p_{n, X})^k \le 1/2$,
    we have $N_X \le kn$.
    If $p_{n, X}\ge 1/2$ holds,
    then $N_X\le n$ immediately holds.
    Thus it suffices to prove $k\le\left\lceil\frac{1-p_{n, X}}{p_{n, X}}\right\rceil$
    when $p_{n, X}<1/2$.
    Indeed, by the motonicity of $(1+1/x)^x$ over $x>0$,
    we have
    \begin{align*}
        \left(\frac1{1-p_{n, X}}\right)^{\frac{1-p_{n, X}}{p_{n, X}}}
        =\left(1+\frac{p_{n, X}}{1-p_{n, X}}\right)^{\frac{1-p_{n, X}}{p_{n, X}}}
        \ge 2,
        \tag{$\because\ p_{n, X} < 1/2$}
    \end{align*}
    so the conclusion follows.
\end{proof}

\begin{rem}
    Although the estimate $N_X\le \frac{n}{p_{n, X}}$ looks loose in general,
    $N_X\le \frac{2d}{p_{2d, X}}$ is a sharp uniform bound for each dimension $d$
    up to a universal constant.
    Indeed, in Example \ref{eg-1dim} and Example \ref{eg-multi} (Appendix \ref{app-ex}),
    we prove that
    \[
        \lim_{\ve\searrow 0}\sup_{\substack{\text{$X$:$d$-dimensional}\\ p_{2d, X} < \ve}} \frac{N_Xp_{2d, X}}{2d} \ge \frac14
    \]
    holds for each positive integer $d$.
    In contrast, the other inequality $p_{n, X}\le\binom{n}{m}p_{m, X}$
    is indeed very loose and drastically improved in Proposition \ref{improved-relation}.
\end{rem}

In Proposition \ref{prop0} and  \ref{naive},
we have never used the information of dimension except for observing $p_{d+1, X}>0$
in Proposition \ref{prop0}.
However, when the distribution of $X$ has a certain regularity,
there already exists a strong result that reflects the dimensionality.
\begin{thm}[{\citealt{wag01}}]\label{thm-wag}
    When the distribution of $X$ is absolutely continuous
    with respect to the Lebesgue measure on $\R^d$,
    \begin{equation}
        p_{n, X}
        \le 1- \frac1{2^{n-1}}\sum_{i=0}^{d-1}\binom{n-1}{i}\label{wag}
        = \frac1{2^{n-1}}\sum_{i=0}^{n-d-1}\binom{n-1}{i}
    \end{equation}
    holds for each $n\ge d+1$.
    The equality is attained if and only if the distribution is balanced,
    i.e., $\P{\ip{c, X} \le 0}=1/2$ holds for all the unit vectors $c\in\R^d$.
\end{thm}

The authors of \citet{wag01} derived this result by showing the existence of
a nonnegative continuous function $h_X$ on $[0,1]$ such that
$h_X(t)=h_X(1-t)$, $h_X(t)\le \frac{d+1}2\min\{t^d, (1-t)^d\}$ and 
\begin{equation}
    p_{n,X}=2\binom{n}{d+1}\int_0^1 t^{n-d-1}h_X(t)\dd t.
    \label{moment}
\end{equation}
We shall provide an intuitive description of the function $h_X$.
Let us consider a one-dimensional i.i.d. sequence $Y_1, Y_2, \ldots$ (also independent from $X_1, X_2, \ldots$),
where each $Y_i$ follows the uniform distribution over $(0, 1)$.
If we consider the $(d+1)$-dimensional random vectors $\tilde{X}_i:=(X_i, Y_i)$,
then, for each $n$, $0\in\cv\{X_1,\ldots,X_n\}\subset\R^d$ is obviously equivalent to the condition
that the $(d+1)$-th coordinate axis (denoted by $\ell$) intersects
the convex set $\tilde{C}_n:=\cv\{\tilde{X}_1,\ldots,\tilde{X}_n\}\subset\R^{d+1}$.

Under a certain regularity condition,
there are exactly two facets (a $d$-dimensional face of $C_n$)
respectively composed of a $(d+1)$-point subset of $\{\tilde{X}_1,\ldots,\tilde{X}_n\}$
that intersects $\ell$.
Let us call them {\it top} and {\it bottom},
where the top is the facet whose intersection with $\ell$
has the bigger $(d+1)$-th coordinate.
Let us define another random variable $H$ as
\begin{itemize}
    \item $0$ if $\ell$ does not intersect $\cv\{\tilde{X}_1,\ldots, \tilde{X}_{d+1}\}$,
    \item   otherwise
    the probability that
        $0$ and $\tilde{X}_{d+2}$
        are on the same side of the hyperplane supporting
        $\cv\{\tilde{X}_1,\ldots,\tilde{X}_{d+1}\}$
        (conditioned by $\tilde{X}_1, \ldots, \tilde{X}_{d+1}$).
\end{itemize}
Then, for a given realization of $\{\tilde{X}_1,\ldots,\tilde{X}_n\}$,
the probability that $\cv\{\tilde{X}_1,\ldots,\tilde{X}_{d+1}\}$
becomes the top of $\tilde{C}_n$ is $H^{n-d-1}$.
As there are $\binom{n}{d+1}$ choice of (equally) possible ``top,"
we can conclude that
\[
    p_{n, X} = \P{\text{$\ell$ intersects $\tilde{C}_n$}}
    =\binom{n}{d+1}\P{\text{$\{X_1,\ldots,X_{d+1}\}$ is the top of $\tilde{C}_n$}}
    =\binom{n}{d+1}\E{H^{n-d-1}}.
\]
A similar observation shows $p_{n, X} = \binom{n}{d+1}\E{(1-H)^{n-d-1},\ H>0}$,
and so we can understand $h_X$ as the density of a half mixture of $H$ and $1-H$ over $\{H>0\}$.
This has been a simplified explanation of $h_X$.
For more rigorous arguments and proofs, see \citet{wag01}.

By using this ``density" function,
we can prove the following interesting relationship.
\begin{prop}\label{improved-relation}
    Let $X$ be an $\R^d$-valued random variable with an absolutely continuous distribution.
    Then, for any integers $n\ge m \ge d+1$,
    we have
    \begin{equation}
        \frac1{2^{n-m}}\frac{n(n-1)\cdots(n-d)}{m(m-1)\cdots(m-d)}p_{m,X}
        \le p_{n, X}
        \le\frac{n(n-1)\cdots(n-d)}{m(m-1)\cdots(m-d)}p_{m,X}.
        \label{relative}
    \end{equation}
\end{prop}

\begin{proof}
    The right inequality is clear from \eqref{moment}.
    For the left inequality,
    by using $h_X(t)=h_X(1-t)$,
    we can rewrite \eqref{moment} as
    \[
        p_{n, X}=\binom{n}{d+1}\int_0^1 t^{n-d-1}(h_X(t)+h_X(1-t))\dd t
        =\binom{n}{d+1}\int_0^1 (t^{n-d-1}+(1-t)^{n-d-1})h_X(t)\dd t.
    \]
    We can prove for $a\ge b\ge 0$ that
    $\frac{t^a+(1-t)^a}{t^b+(1-t)^b}$ attains its minimum at $t=1/2$,
    e.g., by using the method of Lagrange multipliers.
    Accordingly, we obtain
    \begin{align*}
        \frac{p_{n, X}}{\binom{n}{d+1}}&=\int_0^1 (t^{n-d-1}+(1-t)^{n-d-1})h_X(t)\dd t\\
        &\ge 2^{m-n}\int_0^1(t^{m-d-1}+(1-t)^{m-d-1})h_X(t)\dd t
        =2^{m-n}\frac{p_{m,X}}{\binom{m}{d+1}},
    \end{align*}
    which is equivalent to the inequality to prove.
\end{proof}

\begin{rem}\label{rem-important}
    The left inequality has nothing to say when $n$ and $m$ are large
    so $2^{n-m}$ is faster than $(n/m)^d$.
    However, for small $n$ and $m$, it works as a nice estimate.
    Consider the case $n=2d$ and $m=d+1$.
    Then, the proposition and the usual estimate for central binomial coefficients yield
    \[
        p_{2d, X} \ge \frac1{2^{d-1}}\binom{2d}{d+1} p_{d+1, X}
        \ge \frac1{2^{d-1}}\left(\frac{d}{d+1}\frac{2^{2d}}{2\sqrt{d}}\right)p_{d+1, X}
        = \frac{2^d\sqrt{d}}{d+1}p_{d+1, X}.
    \]
    This is comparable to the symmetric case,
    where $p_{d+1,X}=1/2^d$ and $p_{2d, X}=1/2$ hold.
    
    The right inequality is an obvious improvement of the dimension-free estimate
    given in Proposition \ref{naive}.
\end{rem}

We next generalize these results to general distributions including discrete ones
such as empirical measures.
However, at least we have to assume $p_{d, X}=0$.
Note that it is weaker than the condition that $X$ has an absolutely continuous distribution,
as it is satisfied with usual empirical measures (see Proposition \ref{empirical}).

From smoothing arguments, we obtain the following generalization of
inequalities \eqref{wag} and \eqref{relative}.
\begin{thm}\label{smoothing}
    Let $X$ be an arbitrary $d$-dimensional random vector with $p_{d, X}=0$.
    Then, for any $n\ge m\ge d+1$, inequalities
    \[
        p_{n, X}
        \le 1- \frac1{2^{n-1}}\sum_{i=0}^{d-1}\binom{n-1}{i},\qquad
        \frac1{2^{n-m}}\frac{\binom{n}{d+1}}{\binom{m}{d+1}}p_{m,X}
        \le p_{n, X}
        \le\frac{\binom{n}{d+1}}{\binom{m}{d+1}}p_{m,X}
    \]
    hold.
\end{thm}

\begin{proof}
    Let $U$ be a uniform random variable over the unit ball of $\R^d$
    which is independent from $X$.
    Let also $U_1, U_2 \ldots$ be independent copies of $U$, which is independent from
    $X_1, X_2, \ldots$.
    We shall prove that $\lim_{\ve\searrow0}p_{n, X+\ve U} = p_{n, X}$ for each $n$.
    Note that the distribution of $X+\ve U$ has the probability density function
    \[
        f(x)=\frac1{V\ve^d}\P{\|X-x\|_2\le\ve},
    \]
    where $V$ denotes the volume of the unit ball.
    Therefore, once we establish the limit $\lim_{\ve\searrow0}p_{n, X+\ve U} = p_{n, X}$
    the statement of the theorem is clear.
    
    From $p_{d, X}=0$,
    we know that
    \begin{equation}
        q_X(\delta):=\P{\inf_{y\in\cv\{X_i\}_{i=1}^d}\|y\| \le \delta}\to 0,
        \qquad \delta\searrow 0.
        \label{d-delta}
    \end{equation}
    For each $n\ge d+1$, consider the event
    $A_n:=\{0\in\cv\{X_1,\ldots,X_n\}\}$.
    If the closed $\ve$-ball centered at $0$ is included in $\cv\{X_1,\ldots,X_n\}$,
    then $0$ is also contained in $\cv\{X_i+\ve U_i\}_{i=1}^n$ as $\|\ve U_i\|\le\ve$ for all $i$
    (more precisely, we can prove this by using the separating hyperplane theorem).
    Therefore, by considering the facets of the convex hull, we have
    \[
        \P{A_n\cap \bigcap_{\substack{J\subset\{1,\ldots,n\}\\ \lvert J \rvert =d}}
        \left\{\inf_{y\in\cv\{X_i\}_{i\in J}}\|y\| \ge \ve \right\}}
        \le \P{0\in \cv\{X_i+\ve U_i\}_{i=1}^n} = p_{n, X+\ve U}.
    \]
    By using \eqref{d-delta}, we have
    \begin{align*}
        p_{n, X+\ve U} &\ge \P{A_n} -
        \P{\bigcup_{\substack{J\subset\{1,\ldots,n\}\\ \lvert J \rvert =d}}
        \left\{\inf_{y\in\cv\{X_i\}_{i\in J}}\|y\| < \ve \right\}} \\
        &\ge p_{n, X} - \binom{n}{d}q_X(\ve) \to p_{n, X} \qquad (\ve \searrow 0),
    \end{align*}
    and so we obtain $\liminf_{\ve\searrow 0} p_{n, X+\ve U} \ge p_{n, X}$.
    
    On the other hand, if we have $0\in\cv\{X_i+\ve U_i\}_{i=1}^n$ and
    $0\not\in \cv\{X_i\}_{i=1}^n$ at the same time,
    then there exsits $J\subset\{1,\ldots,n\}$ such that $\lvert J\rvert=d$
    and $\inf_{y\in\cv\{X_i\}_{i\in J}}\|y\|\le\ve$.
    Indeed, we can write $0$ as a convex combination
    $\sum_{i=1}^n \lambda_i(X_i+\ve U_i) =0$,
    so
    \[
        \left\|\sum_{i=1}^n\lambda_i X_i\right\|= \left\| \ve\sum_{i=1}^n\lambda_i U_i\right\|
        \le \ve \sum_{i=1}^n\lambda_i\|U_i\| \le \ve.
    \]
    As $0\not\in\cv\{X_i\}_{i=1}^n$,
    there is a facet within $\ve$-distance from $0$.
    Therefore, we obtain
    \[
        \P{0\in\cv\{X_i+\ve U_i\}_{i=1}^n} \le \P{A_n\cup
            \bigcup_{\substack{J\subset\{1,\ldots,n\}\\ \lvert J \rvert =d}}
            \left\{ \inf_{y\in\cv\{X_i\}_{i\in J}}\|y\| \le \ve \right\}
        },
    \]
    and similarly it follows that
    \[
        p_{n, X+\ve U}\le p_{n, X} + \binom{n}{d}q_X(\ve)\quad 
        and
        \quad \limsup_{\ve\searrow0}p_{n,X+\ve U}\le p_{n, X}.
    \]
    Thus we finally obtain $\lim_{\ve\searrow0}p_{n, X+\ve U} = p_{n, X}$.
\end{proof}

We should remark that $p_{d, X}=0$
is naturally satisfied with (centerized) empirical measures.
\begin{prop}\label{empirical}
    Let $\mu$ be an absolutely continuous probability distribution on $\R^d$
    and $Y_1, Y_2, \ldots$ be an i.i.d. samplings from $\mu$.
    Then, with probability one, for each $M\ge d+1$, distributions
    \[
        \mu_M:=\frac1M\sum_{i=1}^M\delta_{Y_i} \quad and
        \quad \tilde{\mu}_M:=\frac1M\sum_{i=1}^M\delta_{Y_i-\frac1M\sum_{j=1}^M Y_j}
    \]
    satisfy $p_{d, \mu_M} = p_{d, \tilde{\mu}_M} = 0$.
    $p_{d, \mu_M} = 0$ also holds for $1\le M\le d$ and
    requires only $p_{d, \mu} = 0$.
\end{prop}
\begin{proof}
    For $\mu_M$, it suffices to prove that with probability one
    there are no $J\subset\{1,\ldots,M\}$ with $\lvert J\rvert = d$
    such that $0\in \cv\{Y_i\}_{i\in J}$.
    This readily follows from the absolute continuity of the original measure $\mu$.
    The extension to the case $\mu$ satisfies only $p_{d,\mu}=0$ is immediate.
    
    For the centered version $\tilde{\mu}_M$,
    what to prove is that with probability one
    there are no $J\subset\{1,\ldots,M\}$ with $\lvert J\rvert = d$
    such that
    $\frac1M\sum_{i=1}^M Y_j \in \cv\{Y_i\}_{i\in J}$.
    Suppose this occurs for some $J$.
    Then, we have that $\frac1{M-d}\sum_{i\ne J}Y_i$ is on the affine hull of
    $\{Y_i\}_{i\in J}$.
    However, as $\{Y_i\}_{i\not\in J}$ is independent from $\{Y_i\}_{i\in J}$
    for a fixed $J$, this probability is zero again from the absolute continuity of $\mu$.
    Therefore, we have the desired conclusion.
\end{proof}


\section{Uniform bounds of $p_{n, X}^\ve$ via the relaxed Tukey depth}\label{sec:ep}
We have not used any quantitative assumption
on the distribution of $X$ in the previous section.
In this section, however,
we shall evaluate $p_{n, X}$ and its $\ve$-approximation version
by using the Tukey depth and its relaxation.
We shall fix an arbitrarily real inner product $\ip{\cdot, \cdot}$ on $\R^d$,
and use the induced norm $\|\cdot\|$
and the notation $\dist(x, A):=\inf_{a\in A}\|x - a\|$
for an $x\in\R^d$ and $A\subset\R^d$.

For a $d$-dimensional random vector $X$ and $\theta\in\R^d$,
define an $\ve$-relaxation version of the Tukey depth by
\[
    \alpha^\ve_X(\theta)
    :=\inf_{\|c\| = 1}\P{\ip{c, X - \theta} \le \ve}.
\]
We also define, for a positive integer $n$,
\[
    p^\ve_{n, X}(\theta)
    :=\P{\dist(\theta, \cv\{X_1, \ldots, X_n\}) \le \ve},
\]
where $X_1, \ldots, X_n$ are independent copies of $X$.
Note that $p_{n, X} = p_{n, X}^0$.
Although we regard them as functions of $\theta$ in Section \ref{new:sec:int},
we only treat the case $\theta = 0$
and omit the argument $\theta$ in this section.

\begin{prop}\label{new:prop:ineq}
    Let $X$ be a $d$-dimensional random vector
    with an absolutely continuous distribution
    with respect to the Lebesgue measure.
    Then, for each $\ve \ge 0$
    and positive integer $n \ge d + 1$,
    we have
    \[
        1 - p_{n, X}^\ve \le \frac{n(1 - \alpha^\ve_X)}{n - d}(1 - p_{n - 1, X}^\ve).
    \]
\end{prop}

Before going into details of quantitative results,
we note the following equivalence of the positivity
of $\alpha_X^\ve$ and $p_{n, X}^\ve$ which immediately follows from this assertion.
\begin{prop}\label{equivalence}
    Let $X$ be an arbitrary $d$-dimensional random vector
    and let $\ve\ge0$.
    Then, $p_{n, X}^\ve > 0$ for some $n\ge1$ implies $\alpha_X^\ve > 0$.
    Reciprocally, $\alpha_X^\ve > 0$ implies $p_{n, X}^\ve > 0$ for all $n\ge d+1$.
\end{prop}
\begin{proof}
    If $\dist(0, \cv\{X_i\}_{i=1}^n)\le\ve$,
    there exists a point $x\in\cv\{X_i\}_{i=1}^n$ with $\|x\|\le\ve$.
    Then, for each $c\in\R^d$ with $\|c\|=1$,
    we have $\ip{c, x}\le \ve$
    and so $\ip{c, X_i}\le \ve$ for at least one $i\in\{1,\ldots,n\}$.
    Hence we have a uniform evaluation
    \[
        \P{\ip{c, X}\le\ve}
        =\frac1n\sum_{i=1}^n\P{\ip{c, X_i}\le\ve}
        \ge \frac1n\P{\bigcup_{i=1}^n\{\ip{c, X_i}\le\ve\}}
        \ge \frac1n\P{0\in\cv\{X_i\}_{i=1}^n}=\frac{p_{n, X}^\ve}n,
    \]
    and the first assertion follows.
    
    For the latter,
    if $\alpha_X^\ve$ is positive,
    we have $p_{n, X}^\ve > 0$ for a sufficiently large $n$
    from Proposition \ref{new:prop:ineq}.
    Finally, Carath\'{e}odory's theorem yields the positivity
    for all $n\ge d+1$.
\end{proof}

Let us prove Proposition \ref{new:prop:ineq}.
\begin{proof}[Proof of Proposition \ref{new:prop:ineq}]
    Let $m\ge d$ be an integer.
    We first consider the quantity $q_m:=1-p_{n, X}^\ve$.
    Let $A_m$ be the event given by
    \[
        \P{\dist(0, \cv\{X_i\}_{i=1}^n) > \ve}.
    \]
    Also, let $B_m$ be the event that $\{X_1, \ldots, X_m\}$ is in general position.
    Then, we have $\P{B_m} = 1$ and $q_m = \P{A_m\cap B_m}$.
    
    Under the event $A_m\cap B_m$,
    we have a unique point $h_m\in \cv\{X_i\}_{i=1}^m$
    that minimizes $\|h_m\|$.
    Let $H_m$ be the open halfspace defined by
    $H_m:=\{x\in\R^d\mid \ip{x - h_m, h_m} > 0\}$.
    Then, the boundary $\partial H_m$
    is the hyperplane going through $h_m$ and perpendicular to $h_m$.
    From the general-position assumption,
    there are at most $d$ points on $\partial H_m$.
    Let $I_m$ be the set of indices $i$ satisfying $\partial H_m$,
    then $I_m$ is a random subset of $\{1, \ldots, m\}$
    with $1\le |I_m|\le d$ under the event $A_m\cap B_m$.
    Note also that $X_i \in H$ for each $i\in\{1, \ldots, d\}\setminus I_m$.
    For simplicity, define $I_m = \emptyset$ for the event $(A_m\cap B_m)^c$. 
    
    As $I_m$ is a random set determined uniquely,
    we can decompose the probability $\P{A_m\cap B_m}$ as follows by symmetry:
    \[
        q_m = \P{A_m\cap B_m}
        =\sum_{k = 1}^d \binom{m}{k}\P{I_m = \{1,\ldots, k\}}.
    \]
    Hence, we want to evaluate the probability $\P{I_m = \{1, \ldots, k\}}$.
    Note that we can similarly define $h_k$ as the unique point in $\cv\{X_i\}_{i=1}^k$
    that minimizes the distance from the origin.
    Then, $H_k$ is the open halfspace
    $H_k=\{x\in\R^d\mid \ip{x - h_{m, k}, h_{m, k}} > 0\}$.
    Then, we have
    \begin{align*}
        \P{I_m = \{1, \ldots, k\}}
        &=\E{
            \mathbbm{1}_{\{\|h_k\| > \ve,\ \cv\{X_i\}_{i=1}^k\subset \partial H_k\}}
            \prod_{j=k+1}^m\P{X_j\in H_k \mid \cv\{X_i\}_{i=1}^k}}\\
        &=\E{
            \mathbbm{1}_{\{\|h_k\| > \ve,\ \cv\{X_i\}_{i=1}^k \subset \partial H_k\}}
            \P{X^\prime\in H_k \mid \cv\{X_i\}_{i=1}^k}^{m - k}},
    \end{align*}
    where $X^\prime$ is a copy of $X$ independent from $X_1, X_2, \ldots$.
    As $\P{X^\prime\in H_k \mid \cv\{X_i\}_{i=1}^k} \le 1 - \alpha_X^\ve$
    under the event $\{\|h_k\| > \ve,\ \cv\{X_i\}_{i=1}^k\subset \partial H_k\}$,
    we have
    \begin{align*}
        \P{I_{m+1} = \{1, \ldots, k\}}
        &= \E{
            \mathbbm{1}_{\{\|h_k\| > \ve,\ \cv\{X_i\}_{i=1}^k\subset \partial H_k\}}
            \P{X^\prime\in H_k \mid \cv\{X_i\}_{i=1}^k}^{m + 1 - k}} \\
        &\le (1 - \alpha_X^\ve)\P{I_m = \{1, \ldots, k\}}.
    \end{align*}
    Therefore, we have
    \begin{align*}
        q_{m+1} &= \sum_{k = 1}^d \binom{m + 1}{k}\P{I_{m+1} = \{1, \ldots, k\}} \\
        &=\sum_{k = 1}^d \frac{m+1}{m+1-k}\binom{m}{k}(1-\alpha_X^\ve)
        \P{I_m = \{1, \ldots, k\}}\\
        &\le\frac{(m+1)(1-\alpha_X^\ve)}{m+1-d}q_m.
    \end{align*}
    By letting $n = m+1$, we obtain the conclusion.
\end{proof}

If we define $g_{d, n}(\alpha)$
by $g_{d, n} := 1$ for $n = 1,\ldots, d$
and
\begin{equation}
    g_{d, n}(\alpha) := \min\left\{1, \frac{n(1-\alpha)}{n-d}g_{d, n-1}(\alpha)\right\}
    \label{new:gdn}
\end{equation}
for $n = d+1, d+2, \ldots$,
we clearly have $1-p_{n, X}^\ve \le g_{d, n}(\alpha_X^\ve)$
from Proposition \ref{new:prop:ineq}
for a $d$-dimensional $X$ having density.
We can actually generalize this to a general $X$.
\begin{lem}\label{new:lem1}
    Let $X$ be an arbitrary $d$-dimensional random vector.
    Then, for each $\ve\ge0$ and positive integer $n$,
    we have $1 - p_{n, X}^\ve \le g_{d, n}(\alpha_X^\ve)$.
\end{lem}
\begin{proof}
    Note first that $g_{d, n}(\alpha)$ is non-increasing with respect to $\alpha\in[0, 1]$.
    Let $\tilde{X}$ be a $d$-dimensional random vector
    such that $\|X-\tilde{X}\|\le\delta$ for some $\delta>0$.
    Then, for an arbitrary $c\in\R^d$ with $\|c\|=1$,
    we have
    \[
        \nip{c, \tilde{X}} \le \ip{c, X} + \delta,
    \]
    so $\P{\ip{c, X}\le \ve} \le \mathbb{P}(\nip{c, \tilde{X}}\le \ve + \delta)$.
    Hence we have $\alpha_X^\ve \le \alpha_{\tilde{X}}^{\ve+\delta}$.
    
    Consider generating
    i.i.d. pairs $(X_1, \tilde{X}_1), \ldots, (X_n, \tilde{X}_n)$
    that are copies of $(X, \tilde{X})$.
    Then, for each $x\in\cv\{X_i\}_{i=1}^n$,
    there is a convex combination such that $x = \sum_{i=1}^n\lambda_iX_i$
    with $\lambda_i\ge0$ and $\sum_{i=1}^n\lambda_i=1$.
    Then, we have
    \[
        \left\|x - \sum_{i=1}^n\lambda_i\tilde{X}_i\right\|
        \le \sum_{i=1}^n\lambda_i\|X_i-\tilde{X}_i\|
        \le \delta.
    \]
    It means that $\inf_{y\in\cv\{\tilde{X}_i\}_{i=1}^n}\|x - y\|\le\delta$
    holds for every $x\in\cv\{X_i\}_{i=1}^n$,
    and we can deduce that $p_{n, X}^{\ve+2\delta}
    \ge p_{n, \tilde{X}}^{\ve+\delta}$ holds.
    
    In particular,
    we can choose $\tilde{X}$ having density,
    so that we have $1 - p_{n, X}^{\ve+\delta}
    \le g_{d, n}(\alpha_{\tilde{X}}^{\ve+\delta})$.
    Therefore, from the monotonicity of $g_{d, n}$,
    we have
    \[
        1 - p_{n, X}^{\ve+2\delta} \le 1 - p_{n, \tilde{X}}^{\ve+\delta}
        \le g_{d, n}(\alpha_{\tilde{X}}^{\ve+\delta})
        \le g_{d, n}(\alpha_X^\ve).
    \]
    As $\delta>0$ can be taken arbitrarily, we finally obtain
    \[
        1 - p_{n, X}^{\ve} \le g_{d, n}(\alpha_X^\ve)
    \]
    by letting $\delta\to0$.
    The $\delta$-relaxation technique used in this proof
    is a big advantage of introducing $p_{n, X}^\ve$
    extending $p_{n, X}$.
\end{proof}

From this lemma,
we obtain the following general bound.
\begin{prop}\label{new:prop-main}
    Let $X$ be an arbitrary $d$-dimensional random vector.
    Then, for each $\ve\ge0$ and positive integer $n\ge d/\alpha_X^\ve$,
    we have
    \[
        1 - p_{n, X}^\ve \le
        \left(
            \frac{n\alpha_X^\ve}{d}
            \exp\left\{
            \left(
                \frac1{\alpha_X^\ve}\log\frac1{1-\alpha_X^\ve}
            \right)
            \left(
                1 + \alpha_X^\ve - \frac{n\alpha_X^\ve}{d}
            \right)\right\}
        \right)^d.
    \]
\end{prop}
\begin{proof}
    From Lemma \ref{new:lem1},
    it suffices to prove that
    \begin{equation}
        g_{d, n}(\alpha) \le \left(
            \frac{n\alpha}d
            \exp\left\{
                \left(\frac1\alpha\log\frac1{1-\alpha}\right)
                \left(1 + \alpha - \frac{n\alpha}{d}\right)
            \right\}
        \right)^d
        \label{new:ineq-to-prove-thm1}
    \end{equation}
    holds for each $\alpha\in(0, 1)$ and $n\ge d/\alpha$.
    From the definition of $g_{d, n}$ (see \eqref{new:gdn}),
    if we set $n_0:=\lceil d/\alpha \rceil$,
    then we have
    \begin{align*}
        g_{d, n}(\alpha)
        &\le \frac{n(n-1)\cdots n_0}{(n-d)(n-d-1)\cdots(n_0-d)}
        (1-\alpha)^{n-n_0+1}g_{d, n_0-1}(\alpha)\\
        &\le \frac{n(n-1)\cdots (n-d+1)}{(n_0-1)(n_0-2)\cdots(n_0-d)}
        (1-\alpha)^{n-n_0+1}\\
        &\le \left(\frac{n}{n_0-d}\right)^d(1-\alpha)^{n-n_0+1}.
    \end{align*}
    As we know $d/\alpha \le n_0 < d/\alpha+1$ by definition,
    we have
    \[
        g_{d,n}(\alpha) \le \left(\frac{n}{d/\alpha - d}\right)^d(1-\alpha)^{n-\frac{d}\alpha}
        =\left(\frac{n\alpha}d\right)^d(1-\alpha)^{n-\frac{d}\alpha-d}.
    \]
    This is indeed the desired inequality \eqref{new:ineq-to-prove-thm1}.
\end{proof}

\begin{rem}\label{new:rem-main}
    As $\frac1\alpha\log\frac1{1-\alpha}\ge 1$ holds on $(0, 1)$
    for $n\ge \frac{(1+\alpha)d}\alpha$,
    the bound \eqref{new:ineq-to-prove-thm1} yields a
    looser but more understandable variant
    \[
        g_{d, n}(\alpha)
        \le\left( \frac{n\alpha}d\exp\left(
            1 + \alpha - \frac{n\alpha}d
        \right)\right)^d.
    \]
    Note that we have a trivial lower bound of
    $1 - p_{n, X}^\ve \ge (1 - \alpha_X^\ve)^n$,
    which is proven by fixing a separating hyperplane
    between the origin and sample points.
\end{rem}

    For a special choice $n=\lceil 3d/\alpha\rceil$,
    the following is readily available

\begin{thm}\label{2^d}
    Let $X$ be an arbitrary $d$-dimensional random vector.
    Then, for each $\ve\ge0$ and positive integer $n\ge 3d/\alpha_X^\ve$,
    we have
    \[
        p_{n, X}^\ve > 1 - \frac1{2^d}.
    \]
\end{thm}
\begin{proof}
    From Proposition \ref{new:prop-main},
    it suffices to prove
    \begin{equation}
        3\exp\left\{\left(\frac1\alpha\log\frac1{1-\alpha}\right)(\alpha - 2)\right\} < \frac12
        \label{new:ineq-to-prove-2}
    \end{equation}
    for all $\alpha\in(0, 1)$.
    If we let $f(x) = \frac{x-2}x\log\frac1{1-x}$ for $x\in(0, 1)$,
    then we have
    \[
        f'(x) =\frac1{x^2} \left(2\log\frac1{1-x} - \frac{x(2-x)}{1-x}\right)
        = \frac1{x^2} \left(2\log\frac1{1-x} + (1 - x) - \frac1{1-x}\right).
    \]
    If we set $t:=\log\frac1{1-x}$,
    $t$ takes positive reals and we have
    \[
        2\log\frac1{1-x} + (1 - x) - \frac1{1-x}
        = 2t + e^{-t} - e^t = 2(t - \sinh t) < 0.
    \]
    Therefore, it suffices to consider the limit $\alpha\searrow0$.
    In this limit,
    the left-hand side of \eqref{new:ineq-to-prove-2} is equal to
    $3e^{-2}$, which is smaller than $1/2$ since $e>\sqrt{6}$ holds.
\end{proof}

We complete this section with a stronger version of Proposition \ref{new:prop:ineq}
only for $\ve = 0$.
Indeed, by summing up the following inequality,
we can immediately obtain the $\ve = 0$ case in Porposition \ref{new:prop:ineq}.
\begin{prop}\label{key}
    Let $X$ be a $d$-dimensional random vector
    with an absolutely continuous distribution with respect to the Lebesgue measure.
    Then,
    \[
        p_{n+1, X}-p_{n, X}
        \le \frac{n(1-\alpha_X)}{n-d}(p_{n, X} - p_{n-1, X})
    \]
    holds for all $n\ge d+1$,
    where we regard.
\end{prop}

\begin{proof}
    First, observe that $p_{n+1, X}-p_{n, X}=\P{0\in\cv\{X_1,\ldots,X_{n+1}\}\setminus\cv\{X_1,\ldots,X_n\}}$
    for $n\ge d+1$ and independent copies $X_1,X_2,\ldots$ of $X$.
    Assume $0\in\cv\{X_1,\ldots,X_{n+1}\}\setminus\cv\{X_1,\ldots,X_n\}$ holds
    and no $d+1$ points of $\{0, X_1,\ldots,X_{n+1}\}$ lie on the same hyperplane
    (the latter is satisfied almost surely as $X$ is absolutely continuous).
    Then, there exists an expression such that
    \[
        0=\sum_{i=1}^{n+1}\lambda_iX_i,\quad \sum_{i=1}^{n+1}\lambda_i = 1,
        \quad \lambda_i\ge0.
    \]
    Here $0<\lambda_{n+1}<1$ must hold as $0\not\in\cv\{X_1,\ldots,X_n\}$ and $X_{n+1}\ne0$.
    Therefore,
    we can rewrite
    \[
        \frac1{1-\lambda_{n+1}}\sum_{i=1}^n\lambda_iX_i= -\frac{\lambda_{n+1}}{1-\lambda_{n+1}}X_{n+1}
    \]
    and this left-hand side is a convex combination of $\{X_1,\ldots,X_n\}$.
    Therefore, the line $\ell$ passing through $X_{n+1}$ and $0$ intersects $\cv\{X_1,\ldots,X_n\}$
    after $0$ (if directed from $X_{n+1}$ to $0$).
    Also, $\ell$ never intersects $\cv\{X_1,\ldots,X_n\}$ before $0$.
    Indeed, if $\lambda X_{n+1}\in\cv\{X_1,\ldots,X_n\}$ for some $\lambda>0$,
    then $0\in\cv\{\lambda X_{n+1}, -\frac{\lambda_{n+1}}{1-\lambda_{n+1}}X_{n+1}\}
    \subset\cv\{X_1,\ldots,X_n\}$ holds and it contradicts the assumption.
    
    Hence, we can define the first hitting point of $\ell$ and $\cv\{X_1,\ldots,X_n\}$ after $0$.
    More formally, let $P$ be the minimum-normed point in $\ell\cap\cv\{X_1,\ldots,X_n\}$.
    Then, by the general-position assumption, there exists a unique $J\subset\{1,\ldots,n\}$
    with $|J|=d$ such that $P\in\cv\{X_i\}_{i\in J}$
    (more strongly, $P$ is in the relative interior pf $\cv\{X_i\}_{i\in J}$).
    In other words, $\cv\{X_i\}_{i\in J}$
    is the unique facet which intersects $\ell$ first.
    Then, there exists a unique normal vector $c_J$ that defines the hyperplane supporting $\{X_i\}_{i\in J}$,
    i.e.,
    $\ip{c_J, X_i}=1$ for each $i\in J$.
    Since $\ip{c_J, P}=1$ also holds, we have
    $\ip{c_J, X_{n+1}} < 0$.
    We can also prove $\ip{c_J, X_i} > 1$ for each $i\in\{1,\ldots,n\}\setminus J$.
    Indeed, if we have $\ip{c_J, X_j} < 1$ for some $j\in\{1,\ldots,n\}\setminus J$,
    then there are interior points of $\cv\{X_i\}_{i\in J\cup\{j\}}$ that belongs to $\ell$
    and this contradicts the minimality of the norm of $P$.
    
    Therefore, for a fixed $J\subset \{1,\ldots,n\}$
    with $|J|=d$,
    the probability that $0\in\cv\{X_1,\ldots,X_{n+1}\}\setminus\cv\{X_1,\ldots,X_n\}$
    holds and $\cv\{X_i\}_{i\in J}$ becomes the first facet intersecting $\ell$ after $0$ is,
    from the independence,
    \begin{align*}
        &\E{\P{0\in\cv\{X_i\}_{i\in J\cup\{n+1\}}\mid \{X_i\}_{i\in J}}
            \prod_{j\in\{1,\ldots,n\}\setminus J} \P{\ip{c_J, X_j} > 1 \mid \{X_i\}_{i\in J}}
        }\\
        &=
        \E{\P{0\in\cv\{X_i\}_{i\in J\cup\{n+1\}}\mid \{X_i\}_{i\in J}}
        \P{\ip{c_J, X^\prime} > 1 \mid \{X_i\}_{i\in J}}^{n-d}},
    \end{align*}
    where $X'$ is a copy of $X$ independent from $\{X_i\}_{i\ge1}$.
    By symmetry, this $J$ is chosen with equal probability
    given $0\in\cv\{X_1,\ldots,X_{n+1}\}\setminus\cv\{X_1,\ldots,X_n\}$
    (almost surely without overlapping).
    Hence, we obtain
    \[
        p_{n+1, X}-p_{n, X}
        =\binom{n}{d}\E{
            \P{0\in\cv\{X_1,\ldots,X_{d+1}\}\mid \{X_i\}_{i\in I}}
            \P{\ip{c_I, X'} > 1 \mid \{X_i\}_{i\in I}}^{n-d}
        },
    \]
    where $I=\{1,\ldots,d\}$.
    Observe that this representation is still valid for $n=d$.
    From the definition of $\alpha_X$,
    we have $\P{\ip{c_I, X'} > 1 \mid \{X_i\}_{i\in I}}\le 1-\alpha_X$,
    so finally obtain, for $n\ge d+1$,
    \begin{align*}
        &p_{n+1, X}-p_{n, X}\\
        &=\binom{n}{d}\E{
            \P{0\in\cv\{X_1,\ldots,X_{d+1}\}\mid \{X_i\}_{i\in I}}
            \P{\ip{c_I, X'} > 1 \mid \{X_i\}_{i\in I}}^{n-d}
        }\\
        &\le(1-\alpha_X)
        \binom{n}{d}\E{
            \P{0\in\cv\{X_1,\ldots,X_{d+1}\}\mid \{X_i\}_{i\in I}}
            \P{\ip{c_I, X'} > 1 \mid \{X_i\}_{i\in I}}^{n-1-d}
        }\\
        &=(1-\alpha_X)\frac{\binom{n}{d}}{\binom{n-1}{d}}(p_{n, X}-p_{n-1, X})\\
        &=\frac{n(1-\alpha_X)}{n-d}(p_{n, X} - p_{n-1, X}).
    \end{align*}
    This is the desired inequality.
\end{proof}


\section{Bounds of $N_X$ via Berry--Esseen theorem}\label{sec4}

In this section,
we discuss upper bounds of $N_X$ for a centered $X$,
which are of particular interest from
the randomized measure reduction (see Section \ref{cubature}).

We know the following assertion as a consequence of Theorem \ref{2^d}.
\begin{thm}\label{N_X}
    Let $X$ be an arbitrary $d$-dimensional random vector.
    Then, we have
    \[
        \frac1{2\alpha_X} \le N_X \le \left\lceil\frac{3d}{\alpha_X}\right\rceil.
    \]
\end{thm}
\begin{proof}
    The right inequality is an immediate consequence of Theorem \ref{2^d}.
    To prove the left one,
    let $n$ be a positive integer satisfying $\frac1{2n} > \alpha_X$.
    Then, there exists a vector $c\in\R^d\setminus\{0\}$ such that
    $\P{c^\top X\le 0} < \frac1{2n}$.
    Then, for $X_1, X_2,\ldots, X_n$ (i.i.d. copies of $X$),
    we have
    \[
        p_{n, X} = \P{0\in\cv\{X_1,\ldots,X_n\}}
        \le \P{\bigcup_{i=1}^n\{c^\top X_i \le 0\}} \le n\P{c^\top X\le 0} <\frac12.
    \]
    Therefore, $N_X$ must satisfy $\frac1{2N_X} \le \alpha_X$.
\end{proof}
\begin{rem}
    The above theorem states that $1/2\le \alpha_X N_X\le 3d + 1$.
    This evaluation for $\alpha_X N_X$ is indeed tight up to a universal constant.
    For example, if $X$ is a $d$-dimensional standard Gaussian, we have $\alpha_X=\frac12$
    and $N_X=2d$, so $\alpha_XN_X = d$.
    Moreover, for a small $\ve\in(0, 1)$,
    if we consider $X=(X^1, \ldots, X^d)$ such that
    \begin{itemize}
        \item $\P{X^d = 1} = \ve$ and $\P{X^d = -1} = 1- \ve$,
        \item $(X^1, \ldots, X^{d-1})|_{X^d = 1}$ is a standard Gaussian,
        \item $X^1 = \cdots = X^{d-1} = 0$ if $X^d = -1$,
    \end{itemize}
    then we can see $\alpha_X = \ve/2$ and $N_X = \Omega((d-1)/\ve)$
    as $(0, \ldots, 0, 1)$ has to be in the convex hull of samples
    to include the origin in it.
    Hence the bound $\alpha_X N_X = \ord{d}$ is sharp even
    for a small $\alpha_X$.
    
    On the contrary,
    \[
        \inf_{\text{$X$:$d$-dimensional}}\alpha_XN_X\le 2
    \]
    holds (even when requiring $p_{d, X}=0$)
    for each positive integer $d$ from Example \ref{eg-1dim} and Example \ref{eg-multi}
    in the appendix (Section \ref{app-ex}).
\end{rem}

Although Theorem \ref{N_X} has strong generality,
in many situations we have little information about the Tukey depth $\alpha_X$.
Indeed, approximately computing the Tukey depth itself is
an important and difficult problem \citep{cue08,zuo19}.
However, if we limit the argument to a centered $X$,
we can obtain various moment-based bounds as shown below.
In this section,
we use the usual Euclidean norm $\|\cdot\|_2$
given by $\|x\|_2 = \sqrt{x^\top x}$ for simplicity.

Let $X$ be a $d$-dimensional centered random vector whose
covariance matrix $V:=\E{XX^\top}$ is nonsingular.
We also define $V^{-1/2}$ as the positive-definite
square root of $V^{-1}$.
Then, for each unit vector $c\in\R^d$
(namely $\|c\|_2 = 1$),
we have
\begin{equation}
    \E{(c^\top V^{-1/2}X)^2}
    =\E{c^\top V^{-1/2}XX^\top V^{-1/2}c}
    =\E{c^\top c}=1,
    \label{jimei}
\end{equation}
We have the following simple result for a bounded $X$.
\begin{prop}\label{new:bdd}
    Let $X$ be a centered $d$-dimensional random vector
    with nonsingular covariance matrix $V$.
    If $\|V^{-1/2}X\|_2\le B$ holds almost surely for a positive constant $B$,
    then we have
    \[
        \alpha_X \ge \frac1{2B^2},
        \qquad N_X \le \bigl\lceil6dB^2\bigr\rceil.
    \]
\end{prop}
\begin{proof}
    For a one-dimensional random variable $Y$ with $\E{Y} = 0$,
    $\E{Y^2} = 1$
    and $|Y|\le B$,
    we have
    \[
        B\P{Y\le 0} \ge \E{ - \min\{Y, 0\}} =\frac12\E{|Y|}
    \]
    and so
    \[
        \P{Y\le 0} \ge \frac{\E{|Y|}}{2B} \ge \frac{\E{|Y|^2}}{2B^2}
        = \frac1{2B^2}.
    \]
    By observing this inequality for each $Y = c^\top V^{-1/2}X$
    with $\|c\|_2 = 1$,
    we obtain the bound of $\alpha_X$.
    The latter bound then follows from Theorem \ref{N_X}.
\end{proof}

Let us consider the unbounded case.
The Berry--Esseen theorem evaluates the speed of convergence
in the central limit theorem \citep{ber41,ess42}.
The following is a recent result with an explicit small constant.

\begin{thm}[{\citealt{kor12}}]\label{one-be}
    Let $Y$ be a random variable with $\E{Y}=0$, $\E{Y^2}=1$,
    and $\E{|Y|^3}<\infty$,
    and let $Y_1, Y_2, \ldots$ be independent copies of $Y$.
    Also let $Z$ be one-dimensional standard Gaussian.
    Then, we have
    \[
        \left\lvert\P{\frac{Y_1+\cdots+Y_n}{\sqrt{n}} \le x} - \P{Z \le x}\right\rvert
        \le \frac{0.4784\,\E{|Y|^3}}{\sqrt{n}}
    \]
    for arbitrary $x\in\R$ and $n\ge1$.
\end{thm}

We can apply the Berry--Esseen theorem 
for evaluating the probability
$\P{c^\top S_n\le 0}$ from \eqref{jimei},
where $S_n$
is the normalized i.i.d. sum
$\frac1{\sqrt{n}}V^{-1/2}(X_1+\cdots+X_n)$.
By elaborating this idea, we obtain the following bound of $N_X$.

\begin{thm}\label{main-be}
    Let $X$ be a centered $d$-dimensional random vector with nonsingular covariance matrix $V$.
    Then,
    \[
        N_X\le 17d\left(1 + \frac94\sup_{c\in\R^d,\|c\|_2=1}
        \E{\left\lvert c^\top V^{-1/2} X\right\rvert^3}^2\right)
    \]
    holds.
\end{thm}

\begin{proof}
    Let $n$ be an integer satisfying
    \[
        n \ge \frac94\sup_{c\in\R^d,\|c\|_2=1}\E{\left\lvert c^\top V^{-1/2} X\right\rvert^3}^2.
    \]
    Then, for an arbitrary $\|c\|_2=1$, from Theorem \ref{one-be}, we have
    \begin{align*}
        \P{\frac{c^\top V^{-1/2}(X_1+\cdots+X_n)}n\le 0}
        &=\P{\frac{c^\top V^{-1/2}(X_1+\cdots+X_n)}{\sqrt{n}}\le 0}\\
        &\ge \frac12 - \frac23\cdot0.48=\frac9{50},
    \end{align*}
    where $X_1,X_2,\ldots$ are independent copies of $X$.
    Hence $\alpha_{n^{-1}(X_1+\cdots+X_n)}\ge9/50$ holds.
    Then we can use Theorem \ref{N_X} to obtain
    \[
        N_{n^{-1}(X_1+\cdots+X_n)} \le \left\lceil \frac{50}9\cdot 3d\right\rceil \le 17d.
    \]
    Since $N_X \le nN_{n^{-1}(X_1+\cdots+X_n)}$ holds, we have
    \[
        N_X\le 17d\left(1 + \frac94\sup_{c\in\R^d,\|c\|_2=1}\E{\left\lvert c^\top V^{-1/2} X\right\rvert^3}^2\right),
    \]
    which is the desired conclusion.
\end{proof}

\begin{rem}
    The bound in Theorem \ref{main-be} is sharp up to constant as a uniform bound
    in terms of $\E{\left\lvert c^\top V^{-1/2} X\right\rvert^3}$.
    Indeed, if $X$ is $d$-dimensional standard Gaussian,
    then $\E{\left\lvert c^\top V^{-1/2} X\right\rvert^3} = \frac{2\sqrt{2}}{\sqrt{\pi}}$
    holds for all $\|c\|_2=1$ while $N_X = 2d$,
    so we have
    \[
        \sup_{c\in\R^d,\|c\|_2=1}\E{\left\lvert c^\top V^{-1/2} X\right\rvert^3}^{-2}N_X = \frac\pi4d.
    \]
\end{rem}

From Theorem \ref{main-be}, we can also obtain several looser but more tractable bounds.
\begin{cor}\label{cor-nx}
    Let $X$ be a centered $d$-dimensional random vector with nonsingular covariance matrix $V$.
    $N_X$ can be bounded as
    \[
        N_X \le 17d
        \left(1 + \frac94\min\left\{
            \E{\left\| V^{-1/2} X\right\|_2^3}^2,\ 
            \E{\left\| V^{-1/2} X\right\|_2^4}
        \right\}
        \right).
    \]
\end{cor}
\begin{proof}
    From Theorem \ref{main-be},
    it suffices to prove
    \[
        \E{\left\lvert c^\top V^{-1/2} X\right\rvert^3}^2
        \le
            \E{\left\| V^{-1/2} X\right\|_2^3}^2,\
            \E{\left\| V^{-1/2} X\right\|_2^4}
    \]
    for each unit vector $c\in\R^d$.
    The first bound is clear from
    \[
        \left\lvert c^\top V^{-1/2} X \right\rvert
        \le \left\|c\right\|_2
        \left\| V^{-1/2} X\right\|_2 = \left\| V^{-1/2} X\right\|_2.
    \]
    The second bound can also be derived as
    \[
        \E{\left\lvert c^\top V^{-1/2} X\right\rvert^3}^2
        \le \E{\left\lvert c^\top V^{-1/2} X\right\rvert^2}\E{\left\lvert c^\top V^{-1/2} X\right\rvert^4}
        =\E{\left\lvert c^\top V^{-1/2} X\right\rvert^4}
        \le\E{\left\| V^{-1/2} X\right\|_2^4},
    \]
    where we have used the Cauchy--Schwarz inequality.
\end{proof}

\begin{rem}\label{optimality}
    In the order notation,
    the first bound in this corollary states
    \[
        N_X=\mathcal{O}\!\left(d\,\E{\left\|V^{-1/2}X\right\|_2^3}^2\right).
    \]
    This estimate is also sharp up to $\ord{d}$ factor
    in the sense that we can prove
    \[
    \sup\left\{\frac{N_X}{\E{\left\|V^{-1/2}X\right\|_2^3}^2}
        \lmid
        \begin{array}{c}
            \text{$X$ is $d$-dimensional},\ \E{X}=0,\\
            \text{$V=\E{XX^\top}$ is nonsingular},\ \E{\left\| V^{-1/2}X\right\|_2^3}<\infty
        \end{array}
    \right\}
    \ge \frac12
\]
for each positive integer $d$.
For the proof of this fact, see Example \ref{eg-1dim} and Example \ref{eg-multi} in the appendix (Section \ref{app-ex}).
\end{rem}

We finally remark that there are multivariate
versions of the Berry--Esseen theorem \citep{zha18,rai19}
and we can use them to derive a bound of $N_X$ in a different approach
which does not use $\alpha_X$.
However, their bounds only gives the estimate
\begin{equation}
    N_X=\mathcal{O}\!\left(d^{7/2}\E{\left\|V^{-1/2}X\right\|_2^3}^2\right),
    \label{yowai}
\end{equation}
which is far worse than the bounds obtained in Theorem \ref{main-be} and Corollary \ref{cor-nx}.
However, it is notable that this approach from multidimensional Berry--Esseen formulas
is applicable to {\it non-identical} $X_i$'s
if the second and third moments are uniformly bounded,
while the combinatorial approach based on $\alpha_X$ seems to be
fully exploiting the i.i.d. assumption.
Therefore,
we provide the details of this alternative approach in the appendix (Section \ref{app}).

\section{Deterministic interior body of random polytopes}\label{new:sec:int}
For each $\alpha > 0$, define a deterministic set defined by
the level sets of Tukey depth
\[
   K^\alpha(X):=\{ \theta\in\R^d \mid \alpha_X(\theta)\ge \alpha\}.
\]
This set is known to be compact and convex \citep{rou99}.
We can also naturally generalize this set for the $\ve$-relaxation of Tukey depth,
and the generalization also satisfies the following:

\begin{prop}\label{new:conv-relax}
    Let $X$ be a $d$-dimensional random vector.
    Then, for each $\ve\ge0$ and $\alpha > 0$,
    the set
    $\{\theta\in\R^d\mid \alpha_X^\ve(\theta) \ge \alpha\}$
    is compact and convex, and satisfies
    \[
        \{\theta\in\R^d\mid \alpha_X^\ve(\theta) \ge \alpha\}
        \supset
        \{
            \theta\in\R^d \mid
            \dist(\theta, K^\alpha(X)) \le \ve
        \}.
    \]
\end{prop}
\begin{proof}
    We fix $\alpha$ and denote
    \[
        K_\ve = \{\theta\in\R^d\mid \alpha_X^\ve(\theta) \ge \alpha\}.
    \]
    Note that $K_0 = K^\alpha(X)$.
    Let $c\in\R^d$ satisfy $\|c\| = 1$.
    Define $t(c)$ by
    \begin{equation}
        t(c) := \inf\{t\in\R\mid \P{\ip{c, X}\le t} \ge \alpha\}.
        \label{new:tc-inf}
    \end{equation}
    If $t(c) = \infty$, i.e., the right-hand set is empty for some $c$,
    then each set $K_\ve$ is empty.
    $t(c) > -\infty$ is clear from $\alpha > 0$.
    Suppose $t(c)\in\R$ for all $c$.
    From the continuity of probability,
    the infimum can actually be replaced by minimum,
    so we have
    \[
        \P{\ip{c, X - \theta} \le \ve} \ge \alpha
        \quad
        \Longleftrightarrow
        \quad
        \ip{c, \theta} + \ve \ge t(c)
    \]
    for each $\theta\in\R^d$.
    Hence, if $\theta_0\in K_0$
    and $\|\theta - \theta_0\|\le \ve$,
    then we have $\theta \in K_\ve$,
    so we obtain the inclusion statement.
    
    Let us prove that $K_\ve$ is compact and convex.
    Define
    $H_\ve(c):=\{\theta\in\R^d \mid \ip{c, \theta} \ge t(c) - \ve\}$
    for each $c\in\R^d$ with $\|c\| = 1$.
    From \eqref{new:tc-inf}, we have
    $K_\ve = \bigcap_{\|c\| = 1}H_\ve(c)$.
    As $H_\ve(c)$ is closed and convex,
    $K_\ve$ is also closed and convex.
    To prove compactness, we shall prove $K_\ve$ is bounded.
    As $X$ is a random vector, there is an $R > 0$ such that
    $\P{\|X\| \ge R} < \alpha$.
    Then, for each $\theta\in\R^d$ satisfying $\|\theta\|\ge R+\ve$,
    we have
    \[
        \P{\ip{ - \frac{\theta}{\|\theta\|}, X - \theta} \le \ve}
        =
        \P{\ip{ - \frac{\theta}{\|\theta\|}, X} \le \ve - \|\theta\|}
        \le
        \P{\| X \| \ge R}
        < \alpha.
    \]
    Therefore, we have $\|\theta\| < R + \ve$ for each $\theta\in K_\ve$
    and so $K_\ve$ is bounded.
\end{proof}

\begin{rem}
    Note that the inclusion stated in Proposition \ref{new:conv-relax}
    can be strict.
    For example, if $X$ is a $d$-dimensional standard Gaussian,
    $K^\alpha(X)$ is empty for each $\alpha > 1/2$,
    but the $\ve$-relaxation of Tukey depth can be
    greater than $1/2$ for $\ve > 0$.
\end{rem}

From this proposition,
we can naturally generalize the arguments given in this section
to the $\ve$-relaxation case;
natural interior bodies of $\ve$-neighborhood of
$\cv\{X_1, \ldots, X_n\}$ are given by the $\ve$-relaxation of Tukey depth.
However, to keep the notation simple,
we only treat $K^\alpha(X)$
the interior body of usual convex hull in the following.

We next prove that the polar body $\bigl(\tilde{K}^\alpha(X)\bigr)^\circ$
used in \citet{gue19},
which we have introduced in Section \ref{new:intro:int},
is essentially the same as $K^\alpha(X)$ in their setting,
i.e., when $X$ is symmetric.
Recall that $\tilde{K}^\alpha(X)$ is defined as
\[
    \tilde{K}^\alpha(X) = \{s\in\R^d \mid \P{\ip{s, X} \ge 1} \le \alpha\}.
\]
Note that the following proposition is not surprising if we go back to
the original background of $\tilde{K}^\alpha$
\citep{car90}, where $X$ is uniform from some deterministic
convex set,
and recent reseaches on its deep relation to the Tukey depth \citep{nag19}.

\begin{prop}\label{new:equivalence}
    Let $X$ be a $d$-dimensional symmetric random vector.
    Then, for each $\alpha \in (0, 1/2)$,
    we have
    \[
        \{\theta\in\R^d \mid \alpha_X(\theta) > \alpha\}
        \subset
        \bigl(\tilde{K}^\alpha(X)\bigr)^\circ
        \subset K^\alpha(X).
    \]
\end{prop}
\begin{proof}
    Consider the set
    \[
        A^\alpha := \{ s \in \R^d \mid \P{\ip{s, X} \ge 1} < \alpha\}.
    \]
    Then, we clearly have $A^\alpha \subset \tilde{K}^\alpha(X)$
    and so $(A^\alpha)^\circ \supset \bigl(\tilde{K}^\alpha(X)\bigr)^\circ$.
    We first prove that $(A^\alpha)^\circ = K^\alpha(X)$ actually holds.
    From the definition of a polar,
    $\theta\in (A^\alpha)^\circ$ if and only if
    \[
        \P{\ip{s, X} \ge 1} < \alpha \quad \Longrightarrow \quad
        \ip{s, \theta} \le 1
    \]
    holds for each $s\in\R^d\setminus\{0\}$.
    If we represent $s = r^{-1}c$ by $r>0$ and $c\in\R^d$ with $\|c\| = 1$,
    this is equivalent to
    \begin{equation}
        \P{\ip{c, X} \ge r} < \alpha \quad \Longrightarrow \quad
        \ip{c, \theta} \le r
        \label{new:eq:r_c}
    \end{equation}
    for each $r > 0$ and $\|c\| = 1$.
    As we have assumed that $X$ is symmetric and $\alpha < 1/2$,
    \eqref{new:eq:r_c} is still equivalent even if we allow $r$ to rake all reals.
    
    We shall prove that,
    for a fixed $c$,
    \eqref{new:eq:r_c} is equivalent to $\P{\ip{c, X - \theta} \ge 0} \ge \alpha$.
    Indeed, if
    \[
        \P{\ip{c, X - \theta} \ge 0}
        = \P{\ip{c, X} \ge \ip{c, \theta}} < \alpha
    \]
    holds,
    there exists a $\delta > 0$ such that
    $\P{\ip{c, X} \ge \ip{c, \theta} - \delta} < \alpha$.
    Then, we have the negation of \eqref{new:eq:r_c}
    by letting $r = \ip{c, \theta} - \delta$.
    For the opposite direction,
    if we assume $\P{\ip{c, X} \ge \ip{c, \theta}} \ge \alpha$,
    we have $\P{\ip{c, X}\ge r} \ge \alpha$ for all $r < \ip{c, \theta}$
    and so \eqref{new:eq:r_c} is true.
    Therefore, we obtain $(A^\alpha)^\circ = K^\alpha(X)$.
    
    For each $\beta\in (\alpha, 1/2)$,
    we clearly have $\tilde{K}^\alpha(X) \subset A^{\beta}$.
    Therefore, we have
    \[
        \bigcup_{\alpha < \beta < 1/2}
        K^{\beta}(X) \subset
        \bigl(\tilde{K}^\alpha(X)\bigr)^\circ
        \subset K^\alpha(X),
    \]
    which is the desired assertion.
\end{proof}

We are going to prove the extension of Theorem \ref{new:thm:gue}
by finding a finite set of points whose convex hull approximates $K^\alpha(X)$.
The following statement is essentially well-known
\citep{pis99,bar14},
but we give the precise statement and a brief proof for completeness.
\begin{prop}\label{new:known}
    Let $K$ be a compact and convex subset of $\R^d$
    such that $K = -K$.
    Then, for each $\ve \in (0, 1)$,
    there is a finite set $A\subset \R^d$ such that
    \[
        (1-\ve)K \subset \cv A\subset K,
        \qquad
        |A| \le \left(1 + \frac2\ve\right)^d.
    \]
\end{prop}
\begin{proof}
    We can only consider the case $K$ has full dimension,
    i.e., $K$ has a nonempty interior.
    Then, the Minkowski functional of $K$ (e.g., see \citealp[IV.1.14]{con07})
    \[
        \vertiii{x}:= \inf\{t\mid t\ge0,\ x\in tK\}
    \]
    defines a norm on $\R^d$
    (note that all norms are equivalent on $\R^d$).
    For this norm, it is known that
    there is a finite subset $A\subset S$ such that
    $\min_{y \in A}\vertiii{x - y} \le \ve$ for all $x\in B$ and
    $|A| \le (1 + 2/\ve)^d$ \citep[Lemma 4.10]{pis99}.
    It suffices to prove $(1-\ve)K\subset \cv A$.
    Assume the contrary, i.e., let $x_0$ be a point such that $\vertiii{x}\le 1-\ve$
    and $x_0\not\in\cv A$.
    Then, there exists a $(d-1)$-dimensional hyperplane $H\subset\R^d$ such that
    $x_0\in H$ and all the points in $A$ lie (strictly) on the same side as the origin
    with respect to $H$.
    Let $y\in\mathop\mathrm{argmin}_{x\in H}\vertiii{x}$.
    Then, we have $\vertiii{y} \le 1 - \ve$, and $z:=\vertiii{y}^{-1}y$ satisfies
    $\min_{x\in H}\vertiii{z - x} \ge \ve$.
    Hence, we have $\min_{x\in A}\vertiii{z - x} > \ve$ and it contradicts
    the assumption for $A$.
\end{proof}

\begin{thm}\label{new:thm-interior}
    Let $X$ be an arbitrary symmetric $d$-dimensional random vector,
    and let $\alpha, \delta, \ve\in(0, 1)$.
    If a positive integer $n$ satisfies
    \[
        n \ge \frac{2d}\alpha \max\left\{\frac{\log(1/\delta)}d + \log\frac1\ve,\ 6\right\},
    \]
    then we have, with probability at least $1 - \delta$,
    \[
        \cv\{X_1, \ldots, X_n\} \supset (1-\ve)K^\alpha(X),
    \]
    where $X_1, X_2, \ldots$ are independent copies of $X$.
\end{thm}

\begin{proof}
    As $K^\alpha(X)$ is symmetric and convex,
    there is a set $A\subset K^\alpha(X)$ with cardinality at most $(1+2/\ve)^d$
    such that $(1-\ve)K^\alpha(X)\subset \cv A$
    from Proposition \ref{new:known}.
    We shall evaluate the probability of $A\subset \cv\{X_i\}_{i=1}^n$.
    As each point $\theta\in A$ satisfies $\alpha_X(\theta)\ge\alpha$,
    from Remark \ref{new:rem-main},
    we have
    \begin{equation}
        1 - p_{n, X}(\theta)
        \le
        \left(
            \frac{n\alpha}d
            \exp
                \left(1 + \alpha - \frac{n\alpha}{d}\right)
        \right)^d
        \label{new:eval-1}
    \end{equation}
    for each $\theta\in A$.
    Hence, it suffices to prove the right-hand side of \eqref{new:eval-1}
    is bounded by $(1 + 2/\ve)^{-d}\delta$.
    By taking the logarithm, it is equivalent to showing
    \[
        \frac{n\alpha}d - \log \frac{n\alpha}d
        \ge 1 + \alpha + \frac{\log(1/\delta)}d + \log\left(1 + \frac2\ve\right).
    \]
    Let us denote $x:=n\alpha/d$.
    For $x\ge12$, as $x/2-\log x$ is increasing,
    we have
    \[
        \frac{x}2 - \log x \ge 6 - \log 6 \ge 2 + \log 3 \ge 1 + \alpha + \log 3
    \]
    by a simple computation.
    Therefore, from $\log(1+2/\ve) \le \log 3 + \log(1/\ve)$ and the assumption for $n$,
    we obtain the inequality \eqref{new:eval-1}.
\end{proof}

\begin{rem}
    Although the bound given in Theorem \ref{new:thm-interior}
    requires $n\ge 12d/\alpha$,
    it can be loosened for moderate $\delta$ and $\ve$.
    For example, if we want to obtain a bound for the case
    $\delta = \ve = 1/2$, then we can prove $n \ge 5d/\alpha$ to be sufficient
    by using the bound in Proposition \ref{new:prop-main}.
    Moreover, we should note that we have used the assumption that
    $X$ is symmetric only for assuring that
    $K^\alpha(X)$ is symmetric so that we can use Proposition \ref{new:known}.
    If we take a symmetric convex subset $K\subset K^\alpha(X)$,
    we can prove a similar inclusion statement for $K$ even for a nonsymmetric $X$.
\end{rem}

If we want a the generalized version of Theorem \ref{new:thm:gue},
we can prove the following:
\begin{cor}\label{new:cor-main}
    Let $X$ be an arbitrary $d$-dimensional symmetric random vector.
    Let $\beta\in(0, 1)$ and set $\alpha = (en/d)^{-\beta}$.
    Then, there exists an absolute constant $c>0.45$ such that,
    for each integer $n$ satisfying
    $n\ge (12e^\beta)^{1/(1-\beta)}d$, we have
    \[
        \cv\{X_1, \ldots, X_n\} \supset \frac12 K^\alpha(X) 
    \]
    with probability at least $1 - \exp(-ce^{-\beta}
        n^{1-\beta}d^\beta)$,
    where $X_1, X_2, \ldots$ are independent copies of $X$.
\end{cor}
\begin{proof}
    For $\alpha = (en/d)^{-\beta}$,
    we have
    \[
        \frac{\alpha}{12d}n
        =\frac1{12e^\beta}\left(\frac{n}d\right)^{1-\beta},
    \]
    so $n\ge 12d/\alpha$ is equivalent to $n\ge (12e^\beta)^{1/(1-\beta)}d$.
    Hence, from Theorem \ref{new:thm-interior},
    it suffices to determine how small $\delta$ can be taken
    so as to satisfy
    \[
        n \ge \frac{2d}\alpha
        \left(\frac{\log(1/\delta)}d + \log 2\right).
    \]
    As $n\ge 12d$ holds for all $\beta$,
    for $a:=\frac{\log2}6 < 0.1$, we have $an \ge \frac{2d}\alpha\log 2$.
    Therefore,
    we can take $\delta$ as small as
    \[
        \log(1/\delta) = \frac{\alpha}2(1 - a)n =
        \frac{1-a}{2}e^{-\beta} n^{1-\beta}d^\beta.
    \]
    Therefore, we can take $c = \frac{1-a}2 > 0.45$ as desired.
\end{proof}

\section{Application}\label{sec5-1}
We discuss implications of the results of this paper in two parts.
The first part discusses the use of the bounds we gave on $p_{n, X}$,
while the second part gives implication of $N_X$'s bounds
on the randomized cubature construction.

\subsection{Bounds of $p_{n, X}$}
Firstly, the inequality between $p_{n, X}$ and $p_{m, X}$ given
in Proposition \ref{improved-relation} provides the inequality
\begin{equation}
    p_{2d, X} \ge \frac{2^d\sqrt{d}}{d+1}p_{d+1, X}
    \label{imp}
\end{equation}
as is mentioned in Remark \ref{rem-important}.

\paragraph{Measure reduction}
Consider a discrete (probability) measure $\mu = \sum_{x\in \mathcal{X}}w_x\delta_{x}$
for a finite subset of $\mathcal{X}\subset\R^d$.
In \citet{cos20},
randomized algorithms for constructing a convex combination satisfying
$\mathbb{E}_{X\sim\mu}[X] = \sum_{i=1}^{d+1}\lambda_i x_i$
($x_i\in \mathcal{X}$),
whose existence is assured by Tchakaloff's theorem \citep{tch57,bay06},
are considered.
As a basic algorithm, the authors considers
the following scheme:
\begin{itemize}
    \item[(a.1)] Randomly choose $d$ points $A = \{x_1,\ldots,x_d\}$ from $\mathcal{X}$.
    \item[(a.2)] For each $x\in\mathcal{X}\subset A$,
    determine if $\mathbb{E}_{X\sim\mu}[X]\in\cv(A\cup\{x\})$ or not,
    and finish the algorithm and return $A\cup\{x\}$ if it holds.
    \item[(a.3)]  Go back to (a.1).
\end{itemize}
Although we can execute the decision for each $x$ in (a.2) with $\ord{d^2}$ computational cost
with an $\ord{d^3}$ preprocessing for a fixed $A$,
the overall expected computational cost until the end of the algorithm is
at least $\Omega\!\left(d^2/p_{d+1,X}\right)$
under some natural assumption on $\mu$ (see Proposition \ref{empirical}).

However, we can also consdier the following naive procedure:
\begin{itemize}
    \item[(b.1)] Randomly choose $2d$ points $B = \{x_1,\ldots,x_{2d}\}$ from $\mathcal{X}$.
    \item[(b.2)] Return $B$ if $\mathbb{E}_{X\sim\mu}[X]\in\cv B$,
    and go back to (b.1) if not.
\end{itemize}
By using an LP solver with the simplex method
we can execute (b.2) in (empirically) $\ord{d^3}$ time
\citep{pan85,sha87}.
Hence the overall computational cost can be heuristically bounded above by $\ord{d^3/p_{2d,X}}$,
which is faster than the former by $\Omega\!\left( d^{-3/2}2^d \right)$
from the evaluation in \eqref{imp}.
Note also that we have rigorously polynomial bounds via other LP methods
(e.g., an infeasible-interior-point method \citep{miz94}),
and so the latter scheme is preferable even in worst-case
when the dimension $d$ becomes large.

\paragraph{Relation between two depths}
We can also deduce an inequality between two depth concepts in statistics.
As is mentioned in Introduction,
for a random vector $X\in\R^d$,
$p_{d+1, X}$ is called the simplicial depth whereas $\alpha_X$ is the Tukey depth of the origin
with respect to $X$.

Naively, we have
$\alpha_X \ge \frac{p_{n, X}}n$
for each $n$, so $\alpha_X \ge \frac{p_{d+1, X}}{d+1}$ holds.
However, by using \eqref{imp} here,
we obtain a sharper estimate
\[
    \alpha_X \ge \frac{p_{2d, X}}{2d} \ge \frac1{2d}\frac{2^d\sqrt{d}}{d+1}p_{d+1, X}
    \ge \frac{2^{d-1}}{\sqrt{d}(d+1)}p_{d+1, X}.
\]
In contrast, deriving a nontrivial upper bound of $\alpha_X$ in terms of $p_{d+1, X}$
still seems difficult.

\subsection{Bounds of $N_X$}
Secondly, we give applications of the bounds of $N_X$ given
in Section \ref{sec4}.

\paragraph{Random trigonometic cubature}
Consider a $d$-dimensional
random vector
\[
    X = (\cos\theta, \ldots, \cos d\theta)^\top \in \R^d
\]
for a positive integer $d$,
where $\theta$ is a uniform random variable over $(-\pi, \pi)$.
Then, from an easy computation,
we have
$V:= \E{XX^\top} = \frac12I_d$,
and so we obtain
\[
    \|V^{-1/2}X\|^2 \le 2d
\]
almost surely.
Therefore, from Proposition \ref{new:bdd},
we have
\[
    N_X \le 1 + 12d^2.
\]
This example is equivalent to a random construction of
the so-called Gauss--Chebyshev quadrature \citep[Chapter 8]{chebyshev}.
Although we can bound as above the number of observations required in a random construction,
concrete constructions with fewer points are already known.

Deriving a bound for random construction of cubature without any know deterministic construction,
such as cubature on Wiener space \citep{lyo04,hayakawa-CoW},
which is more important,
is still unsolved and left for future work.

\paragraph{Beyond naive cubature construction}
Recall the cubature construction problem described in Section \ref{cubature}.
We consider a random variable of the form $X = \bm{f}(Y)$,
where $Y$ is a random variable on some topological space $\X$
and $\bm{f}=(f_1,\ldots,f_d)^\top : \X \to \R^d$ is a $d$-dimensional vector valued
integrable function.
Our aim is to find points $y_1, \ldots, y_{d+1} \in \X$
and weights $w_1, \ldots, w_{d+1} \ge 0$ whose total is one such that
\begin{equation}
    \E{\bm{f}(Y)} = \sum_{j = 1}^{d + 1}w_j\bm{f}(y_j).
    \label{eq:lp}
\end{equation}
A naive algorithm proposed by \citet{hayakawa-MCCC} was
to generate independent copies $Y_1, Y_2, \ldots$ of $Y$
and choose $y_j$ from these random samples.
Without any knowledge of $N_X$,
the algorithm would be of the form
\begin{itemize}
    \item[(c.1)] Take $k = 2d$.
    \item[(c.2)] Randomly generate $Y_i$ up to $i = k$
    and determine if \eqref{eq:lp} can be satisfied with $y_j\in\{Y_i\}_{i = 1}^k$
    by using an LP solver.
    \item[(c.3)] If we find a solution, stop the algorithm.
    Otherwise, go to (c.2) after replacing $k$ by $2k$.
\end{itemize}
This procedure ends at $k\le 2N_X(\E{X})$ with probability more than half.
We can then heuristically estimate the computational cost by
$\Theta(C(d, N_X(\E{X})))$,
where we denote by $C(d, n)$
the computational complexity of a linear programming
problem finding the solution of \eqref{eq:lp} from $n$ sample points.
Empirically, this is estimated as $\Omega(d^2n)$ 
or more when we use the simplex method \citep{sha87}.

However, our analysis on $N_X$ via the Berry--Esseen bound
tells us the possibility of an alternative
(Algorithm \ref{rccr}).
\begin{algorithm}[H]
	\caption{Randomized cubature construction for recombination}
	\label{rccr}
	\begin{algorithmic}[1]
		\Require{
		   An integer $\ell \ge 2$}
		\Ensure{
		  $(w_1, y_1), \ldots, (w_n, y_n) \in \R_{\ge0}\times \X$
		  satisfying $\sum_{j=1}^nw_j=1$
		  and $\E{X} = \sum_{j = 1}^n w_j\bm{f}(y_j)$
		}
		\Initialize{
		    $x_1, \ldots, x_{\ell d}$, $z_1, \ldots, z_{\ell d}$ : vectors in $\R^d$,
		    $k\leftarrow 0$
		}
		\For{$i = 1, \ldots, \ell d$}
		    \State{Sample $Y_i$}
		    \State{$x_i \leftarrow \bm{f}(Y_i)$}
		\EndFor
		\While{$\E{X}\not\in\cv\{x_1, \ldots, x_{\ell d}\}$}
            \For{$i = 1, \ldots, \ell d$}
                \State{$z_i \leftarrow 0$ (as an $\R^d$ vector)}
            \EndFor
            \For{$j = 2^k, \ldots, 2^{k+1} - 1$}
                \For{$i = 1, \ldots, \ell d$}
                    \State{Sample $Y_{j\ell d + i}$}
                    \State{$z_i \leftarrow z_i + 2^{-k}\bm{f}(Y_{j\ell d + i})$}
                \EndFor
            \EndFor
            \For{$i = 1, \ldots, \ell d$}
                \State{$x_i \leftarrow (x_i + z_i)/2$}
            \EndFor
            \State{$k \leftarrow k + 1$}
		\EndWhile
		\State{Take $x_{i_1}, \ldots, x_{i_{d+1}}$ and $\lambda_1, \ldots, \lambda_{d+1}$
		such that $\E{X} = \sum_{m = 1}^{d+1}\lambda_m x_{i_m}$ by solving an LP}
		\State{\textbf{Return} $(2^{-k}\lambda_m, Y_{j\ell d + i_m})$
		for $(j, m) \in \{0, \ldots, 2^k-1\}\times\{1, \ldots, d+1\}$}
	\end{algorithmic}
\end{algorithm}

Although the pseudocode may seem a little long,
this is just uses $\ell d$
random vectors of the form $n^{-1}(X_1 + \cdots + X_n)$
as the possible vertices of the convex combination,
which is used for deriving bounds of $N_X$ in Section \ref{sec4}.
After executing Algorithm \ref{rccr},
we can use any algorithm for deterministic measures
(typically called recombination; \citealp{lit12,tch15,maa19})
to obtain an actual $d+1$ points cubature rule,
whose time complexity is rigorously bounded by $\ord{kd^3 + 2^kd^2}$
by using the final value of $k$ in the above algorithm.

As we can carry out Algorithm \ref{rccr} within
$\ord{2^k\ell d^2 + kC(d, \ell d)}$,
the overall computational cost is
$\ord{k C(d, \ell d) + 2^k\ell d^2}$.
Then we have heuristically have the bound $\ord{k\ell d^3 + 2^k\ell d^2}$
for a small $\ell$.
By using the number $N = 2^k\ell d$,
which is the number of randomly generated copies of $Y$,
this cost is rewritten as
\[
    \ord{\log(N/\ell d)\ell d^3 + Nd}.
\]
As our bound for $N_X(\E{X})$ in Theorem \ref{main-be}
is applicable for this $N$
because of the use of Berry--Esseen type estimate
($\ell = 17$ is used in the proof),
we can also give an estimate for this alternative algorithm.
If the $N$ is not as large as $\Omega(dN_X(\E{X}))$ for an appropriate choice of $\ell$,
we indeed have a better scheme,
though the comparison itself may be a nontrivial problem in general.
In any event, the fact that we can avoid solving a large LP problem
is an obvious advantage.


\section{Concluding remarks}\label{sec6}
In this paper, we have investigated inequalities regarding $p_{n, X}, N_X$ and $\alpha_X$,
which is motivated from the fields of numerical analysis, data science, statistics
and random matrix.
We generalized the existing inequalities for $p_{n, X}$ in Section \ref{sec2}.
After pointing out that the convergence rate of $p_{n, X}$ is
determined by $\alpha_X$ in Section \ref{sec:ep}
with introduction of $\ve$-relaxation of both quantities,
we proved that $N_X$ and $1/\alpha_X$ are
of the same magnitude up to an $\ord{d}$ factor in Theorem \ref{N_X}.
We also gave estimates of $N_X$ based on the moments of $X$ in Section \ref{sec4}
by using Berry--Esseen type bounds.
Although arguments have been based on whether a given vector is included in the
random convex polytope $\cv\{X_1, \ldots, X_n\}$, in Section \ref{sec4},
we extended our results to the analysis of deterministic convex bodies
included in the random convex hull,
which immediately led to a technical improvement on
a result from the random matrix community.
We finally discussed several implications of our results
on application in Section \ref{sec5-1}.



\bibliography{cite}
\bibliographystyle{abbrvnat}

\renewcommand{\appendixname}{Appendix}
\appendix

\section{Bounds of $N_X$ via Multivariate Berry--Esseen theorem}\label{app}

In this section,
we provide two different estimates of $N_X$.
Although we can prove that the first bound (Section \ref{first})
is strictly stronger than the second one (Section \ref{second}),
we also give the proof of the second as
there seems to be more room for improvement in the second approach
than in the first.

The following first bound is the one mentioned in \eqref{yowai}.
The proof is given in Section \ref{first}.
\begin{thm}\label{thm-first}
    Let $X$ be an $\R^d$-valued random vector which is centered and
    of nonsingular covariance matrix $V$. Then,
    \[
        N_X\le 8d\left(1 + 36d^2(42d^{1/4}+16)^2\E{\left\|V^{-1/2}X\right\|_2^3}^2\right)
    \]
    holds.
\end{thm}

Note that
\[ 
    \E{\left\|V^{-1/2}X\right\|_2^3}^2\ge \E{\left\|V^{-1/2}X\right\|_2^2}^3=d^3
\]
holds so we can ignore the $\mathcal{O}(d)$ term.
In the case $\sup\left\|V^{-1/2}X\right\|_2<\infty$,
we have
\[
    \E{\left\|V^{-1/2}X\right\|_2^3}^2\le \E{\left\|V^{-1/2}X\right\|_2^2\sup\left\|V^{-1/2}X\right\|_2}^2
    =d^2\sup\left\|V^{-1/2}X\right\|_2^2.
\]
Therefore, the following proposition,
which only states $N_X=\tilde{\mathcal{O}}\!\left(d^{15/2}\sup\left\|V^{-1/2}X\right\|_2^2\right)$,
is weaker than Theorem \ref{thm-first}.
However, the approach of proofs is different
and there seems to remain some room for improvement in the proof of Proposition \ref{prop-second},
so we give the proof in Section \ref{second}.
\begin{prop}\label{prop-second}
    Let $X$ be an $\R^d$-valued random vector
    which is centered, bounded and of nonsingular covariance matrix $V$.
    Then, for all $n$ satisfying
    \[
        \frac{n}{(1+\log n)^2} \le 2^{16}100 d^{13/2}\sup\left\| V^{-1/2}X\right\|_2^2,
    \]
    $N_X\le 6dn$ holds.
\end{prop}

\subsection{Multivariate Berry--Esseen bounds}
Before proceeding to the evaluation of $N_X$,
we briefly review multivariate Berry--Esseen type theorems.
The following theorem should be the best known bound with explicit constants
and dependence with respect to the dimension.

\begin{thm}[{\citealt{rai19}}]\label{rai19}
    Let $Y_1,\ldots,Y_n$ be i.i.d. $D$-dimensional independent random vectors
    with mean zero and covariance $I_D$.
    For any convex measurable set $A\subset\R^D$, it holds
    \[
        \left\lvert \P{\frac{Y_1+\cdots+Y_n}{\sqrt{n}}\in A} - \P{Z\in A} \right\rvert
        \le \frac{(42D^{1/4}+16)\E{\|Y_1\|_2^3}}{\sqrt{n}},
    \]
    where $Z$ is a $D$-dimensional standard Gaussian.
\end{thm}
Note that the original statement is not limited to the i.i.d. case.
However, similarly to the other existing Berry--Esseen type bounds,
Theorem \ref{rai19} only gives information about {\it convex} measurable sets.
Thus we cannot use this result directly.
However, Section \ref{first} gives a creative use of Theorem \ref{rai19}.

Unlike the usual Berry--Esseen results,
the next theorem can be used for nonconvex case with reasonable dependence on dimension.
We denote by $\mathcal{W}_2(\mu, \nu)$ the Wasserstein-2 distribution between two probability
measures $\mu$ and $\nu$ on the same domain.
This is defined formally as
\[
    \mathcal{W}_2(\mu, \nu):=\inf_{Y\sim \mu, Z\sim \nu}\E{\|Y-Z\|_2^2},
\]
where the infimum is taken for all the joint distribution $(Y, Z)$
with the marginal satisfying $Y\sim\mu$ and $Z\sim\nu$.
Although it is an abuse of notation,
we also write $\mathcal{W}_2(Y, Z)$ to represent $\mathcal{W}_2(\mu, \nu)$
when $Y\sim\mu$ and $Z\sim\nu$ for some random variables $Y$ and $Z$.

\begin{thm}[{\citealt{zha18}}]\label{zha18}
    Let $Y_1,\ldots,Y_n$ be $D$-dimensional independent random vectors
    with mean zero, covariance $\Sigma$, and $\|Y_i\|_2\le B$ almost surely for each $i$.
    If we let $Z$ be a Gaussian with covariance $\Sigma$, then we have
    \[
        \mathcal{W}_2\left(\frac{Y_1+\cdots+Y_n}{\sqrt{n}}, Z\right)
        \le \frac{5\sqrt{D}B(1+\log n)}{\sqrt{n}}.
    \]
\end{thm}
For a set $A\subset \R^D$ and an $\ve>0$, define 
\[
    A^\ve:=\left\{x\in\R^D\lmid \inf_{y\in A}\|x-y\|_2\le \ve\right\},
    \qquad
    A^{-\ve}:=\left\{
        x\in\R^D\lmid \inf_{y\in A^c} \|x-y\|_2\ge \ve
    \right\}.
\]
By combining the following assertion with Theorem \ref{zha18},
we derive another bound of $N_X$ in Section \ref{second}.
\begin{prop}\label{zhai}
    Let $Y, Z$ be $D$-dimensional random vectors.
    Then, for any measurable set $A\subset\R^d$ and any $\ve>0$,
    the following estimates hold:
    \begin{align*}
        \P{Y\in A} &\le \P{Z\in A^\ve} + \frac{\mathcal{W}_2(Y, Z)^2}{\ve^2}, \\
        \P{Y\in A} &\ge \P{Z\in A^{-\ve}} - \frac{\mathcal{W}_2(Y, Z)^2}{\ve^2}.
    \end{align*}
\end{prop}

\begin{proof}
    This proof is essentially the same as the argument given in the proof
    of \citet[Proposition 1.4]{zha18}.
    Let $(Y', Z')$ be an arbitrary couple of random variables
    such that $Y'\sim Y$ and $Z'\sim Z$.
    Then, we have
    \begin{align*}
        \P{Y'\in A} &= \P{\|Y'-Z'\|_2 <  \ve,\ Y'\in A} + \P{\|Y'-Z'\|_2 \ge \ve,\ Y\in A} \\
        &\le \P{Z'\in A^\ve} + \P{\|Y'-Z'\|_2 \ge \ve}\\
        &\le \P{Z'\in A^\ve} + \frac1{\ve^2}\E{\|Y'-Z'\|_2^2}\tag{by Chebyshev's inequality}.
    \end{align*}
    By taking the infimum of the right-hand side
    with respect to all the possible couples $(Y', Z')$,
    we obtain the former result.
    The latter can also be derived by evaluating
    \begin{align*}
        \P{Z'\in A^{-\ve}} &= \P{\|Y'-Z'\|_2 < \ve,\ Z'\in A^{-\ve}} + \P{\|Y'-Z'\|_2 \ge \ve,\ Z\in A^{-\ve}} \\
        &\le \P{Y'\in A} + \P{\|Y'-Z'\|_2 \ge \ve}\\
        &\le \P{Y'\in A} + \frac1{\ve^2}\E{\|Y'-Z'\|_2^2}
    \end{align*}
    and again taking the infimum.
\end{proof}

\subsection{The first bound}\label{first}
In this section, we prove Theorem \ref{thm-first}.
We shall set $D=d$ and make use of Theorem \ref{rai19}.

First, fix a set $S\subset\R^d$
and consider the set $C(S):=\{x\in\R^d\mid 0\in \cv(S\cup\{x\})\}$.
We can prove this set is convex for any $S$.
Indeed, if $0\in\cv S$, then clearly $C(S)=\R^d$.
Otherwise, $x\in C(S)$ is equivalent to the existence of
some $k\ge0$ and $x_1, \ldots, x_k\in S$, $\lambda>0$, $\lambda_1,\ldots, \lambda_k\ge 0$
such that
\[
    \lambda+\lambda_1+\cdots+\lambda_k=1,\qquad
    \lambda x+ \lambda_1x_1+\cdots+\lambda_kx_k = 0.
\]
Here, $\lambda>0$ comes from the assumption $0\not\in\cv S$.
This occurs if and only if $x$ is contained in the negative cone of $S$,
i.e., $C(S)=\{\sum_{i=1}^k\tilde{\lambda}_ix_i\mid k\ge0,\ \tilde{\lambda}_i\le 0,\ x_i\in S\}$.
In both cases $C(S)$ is convex, so $S_0$ is always convex (and of course measurable).

Let $X$ be an $\R^d$-valued random vector
with mean $0$ and nonsingular covariance $V$.
Suppose $\E{\left\|V^{-1/2}X\right\|_2^3}<\infty$.
Let $X_1,X_2,\ldots$ be independent copies of $X$,
and for a fixed positive integer $n$,
define
\[
    W_i:=\frac{V^{-1/2}X_{(i-1)n+1}+\cdots+V^{-1/2}X_{in}}{\sqrt{n}}
\]
for $i=1,\ldots,2d$.
We also let $Z_1,\ldots, Z_{2d}$ be independent $d$-dimensional standard Gaussian
which is also independent from $X_1,X_2,\ldots$.
Then, by using Theorem \ref{rai19} and the above-mentioned convexity of $C(S)$,
we have
\begin{align*}
    \P{0\in\{W_1,\ldots,W_{2d}\}}
    &=\P{W_1\in C(\{W_2,\ldots,W_{2d}\})}\\
    &\ge\P{Z_1\in C(\{W_2,\ldots,W_{2d}\})} - \frac{(42d^{1/4}+16)\E{\left\|V^{-1/2}X\right\|_2^3}}{\sqrt{n}}\\
    &=\P{0\in\cv\{Z_1,W_2,\ldots,W_{2d}\}}  - \frac{(42d^{1/4}+16)\E{\left\|V^{-1/2}X\right\|_2^3}}{\sqrt{n}}.
\end{align*}
By repeating similar evaluations,
we obtain
\begin{align*}
    &\P{0\in\cv\{W_1,\ldots, W_{2d}\}}\\
    &\ge \P{0\in\cv\{Z_1, W_2,\ldots,W_{2d}\}} - \frac{(42d^{1/4}+16)\E{\left\|V^{-1/2}X\right\|_2^3}}{\sqrt{n}} \\
    &\ge \P{0\in\cv\{Z_1,Z_2,W_3,\ldots, W_{2d}\}} - \frac{2(42d^{1/4}+16)\E{\left\|V^{-1/2}X\right\|_2^3}}{\sqrt{n}}\\
    &\ \, \vdots\\
    &\ge \P{0\in\cv\{Z_1,\ldots,Z_i,W_{i+1},\ldots, W_{2d}\}} - \frac{i(42d^{1/4}+16)\E{\left\|V^{-1/2}X\right\|_2^3}}{\sqrt{n}}\\
    &\ \, \vdots\\
    &\ge \P{0\in\cv\{Z_1,\ldots,Z_{2d}\}} - \frac{2d(42d^{1/4}+16)\E{\left\|V^{-1/2}X\right\|_2^3}}{\sqrt{n}}\\
    &=\frac12 - \frac{2d(42d^{1/4}+16)\E{\left\|V^{-1/2}X\right\|_2^3}}{\sqrt{n}}.
\end{align*}

Therefore, by letting
\[
    n=\left\lceil 36d^2(42d^{1/4}+16)^2\E{\left\|V^{-1/2}X\right\|_2^3}^2\right\rceil,
\]
we have $\P{0\in\cv\{X_1,\ldots,X_{2dn}\}}\ge 1/6$.
Since $(1-1/6)^4<1/2$ holds, we finally obtain $N_X\le 8dn$.


\subsection{The second bound}\label{second}
In this section, we provide a proof of Section \ref{prop-second}
in a different manner from the one given in the previous section.
We set $D=2d^2$ and define $A_d\subset \R^D$ as follows:
\[
    A_d:=\{x=(x_1, \ldots, x_{2d})\in (\R^d)^{2d}\simeq \R^D \mid 0 \in \cv\{x_1,\ldots,x_{2d}\}\subset\R^d\}.
\]
Then, it suffices to find a suitable upper bound of
$\P{Z\in A_d\setminus A_d^{-\ve}}$ for a $D$-dimensional standard Gaussian $Z$
for our purpose.
For an $\ve>0$, $B_{d,\ve}:=A_d\setminus A_d^{-\ve}$ can be explicitly written as
\begin{equation}
    B_{d,\ve}=\left\{x=(x_1,\ldots,x_{2d})\in\R^D\lmid
    \begin{array}{c}
    0\in\cv\{x_i\}_{i=1}^{2d},\\ \exists \tilde{x}=(\tilde{x}_i)_{i=1}^{2d}\in\R^D
    \ \text{s.t.} \|x-\tilde{x}\|_2<\ve,\ 0\not\in\cv\{\tilde{x}_i\}_{i=1}^{2d}
    \end{array}
    \right\}.
    \label{e-useful}
\end{equation}
For a (finite) set $S=\{v_1,\ldots, v_j\}\subset\R^d$,
define the negative box $N(S)\subset\R^d$ by
\[
    N(S):=\{a_1v_1+\cdots+a_jv_j\mid a_i\in[-1, 0]\}.
\]
$N(S)$ is obviously a convex set.
\begin{lem}\label{maji}
    For an arbitrary $x=(x_1,\ldots,x_{2d})\in B_{d,\ve}$,
    there exists an index $k\in\{1,\ldots,2d\}$ such that
    $x_k\in N(\{x_i\mid i\ne k\})\setminus N(\{x_i\mid i\ne k\})^{-\ve\sqrt{2d}}$.
\end{lem}
\begin{proof}
    As $0\in \cv\{x_i\}_{i=1}^{2d}$, there exist nonnegative weights $\lambda_1,\ldots,\lambda_{2d}$ such that
    $\lambda_1x_1+\cdots+\lambda_{2d}x_{2d}=0$ with the total weight one.
    Let $k$ be an index such that $w_k$ is the maximum weight.
    Then, $\lambda_k$ is clearly positive and we have $x_k=\sum_{i\ne k} -\frac{\lambda_i}{\lambda_k}x_i$.
    Therefore, we obtain $x_k\in N(\{x_i\mid i\ne k\})$.
    
    By \eqref{e-useful},
    there exists an $\tilde{x}=(\tilde{x}_i)_{i=1}^{2d}\in\R^D$ such that
    $\sum_{i=1}^{2d}\|x_i-\tilde{x}_i\|_2^2<\ve^2$ and $0\not\in\cv\{\tilde{x}_i\}_{i=1}^{2d}$.
    We can prove that $\tilde{x}_k\not\in N(\{\tilde{x}_i\mid i\ne k\})$.
    Indeed, if we can write $\tilde{x}_k= - \sum_{i\ne k} a_i\tilde{x}_i$ with $a_i\in[0, 1]$,
    then
    \[
        \left(1+\sum_{i\ne k}a_i\right)^{-1}\left(\tilde{x}_k + \sum_{i\ne k} a_i\tilde{x}_i\right)=0
    \]
    is a convex combination and it contradicts the assumption $0\not\in\cv\{\tilde{x}_i\}_{i=1}^{2d}$.
    Therefore, we can take a unit vector $c\in\R^d$ such that
    \begin{equation}
        c^\top\tilde{x}_k > \max\{c^\top y \mid y\in N(\{\tilde{x}_i\mid i\ne k\})\}.
        \label{separating}
    \end{equation}
    Let us assume the closed ball with center $x_k$ and radius $\delta$
    is included in $N(\{x_i\mid i\ne k\})$ for a $\delta>0$.
    Then, if $\delta > \|x_k-\tilde{x}_k\|_2$,
    the closed ball with center $\tilde{x}_k$ and radius $\delta^\prime:=\delta-\|x_k-\tilde{x}_k\|_2$
    is included in $N(\{x_i\mid i\ne k\})$.
    In particular, we have some coefficients $a_i\in[-1,0]$ such that
    $\tilde{x}_k + \delta^\prime c = \sum_{i\ne k} a_ix_i$.
    By the inequality \eqref{separating},
    we have
    \[
        c^\top \tilde{x}_k > c^\top\sum_{i\ne k}a_i \tilde{x}_i
        =c^\top\left(\tilde{x}_k+\delta^\prime c + \sum_{i\ne k}a_i(\tilde{x}_i-x_i)\right),
    \]
    so by arranging
    \[
        \delta^\prime < \sum_{i\ne k} a_i c^\top(x_i-\tilde{x}_i)
        \le \sum_{i\ne k}\|x_i-\tilde{x}_i\|_2.
    \]
    Therefore, from the definition of $\delta^\prime$,
    we obtain
    \[
        \delta < \sum_{i=1}^{2d}\|x_i-\tilde{x}_i\|_2
        \le \left(2d\sum_{i=1}^{2d}\|x_i-\tilde{x}_i\|_2\right)^{1/2}\le \ve\sqrt{2d}
    \]
    by Cauchy-Schwarz and the assumption.
    It immediately implies the desired assertion.
\end{proof}

\begin{prop}
    $\P{Z\in B_{d,\ve}} \le 8\sqrt{2}d^{7/4}\ve$ holds.
\end{prop}
\begin{proof}
    By Lemma \ref{maji}, we have
    $B_{d,\ve}\subset\bigcup_{k=1}^{2d}\{x\mid x_k\in N(\{x_i\mid i\ne k\})
    \setminus N(\{x_i\mid i\ne k\})^{-\ve\sqrt{2d}}\}$.
    Therefore, letting $Z=(Z_1,\ldots,Z_{2d})$ be a standard Gaussian in $\R^D$
    (where each $Z_i$ is a independent standard Gaussian in $\R^d$),
    we can evaluate
    \[
        \P{Z\in B_{d,\ve}}
        \le \sum_{k=1}^{2d}\P{Z_k\in N(\{Z_i\mid i\ne k\})
        \setminus N(\{Z_i\mid i\ne k\})^{-\ve\sqrt{2d}}\}}.
    \]
    For each $k$, $Z_k$ is independent from the random convex set $N(\{Z_i\mid i\ne k\})$.
    Therefore, we can use the result of \citet{bal93} to deduce
    $\P{Z_k\in N(\{Z_i\mid i\ne k\})\setminus N(\{Z_i\mid i\ne k\})^{-\ve\sqrt{2d}}\}}\le 4d^{1/4}\cdot\ve\sqrt{2d}$.
    Therefore, we finally obtain
    \[
        \P{Z\in B_{d,\ve}} \le 2d \cdot 4d^{1/4} \cdot \ve \sqrt{2d}
        =8\sqrt{2}d^{7/4}\ve.
    \]
\end{proof}

By letting $\ve=2^{-13/2}d^{-7/4}$,
we have $\P{Z\in B_{d, \ve}}\le 1/8$.
Under this value of $\ve$, if we let $n$ satisfy
\begin{equation}
    \frac{n}{(1+\log n)^2} \ge \frac{8\cdot 25 DB^2}{\ve^2}
    =400d^2B^2 \cdot 2^{13}d^{7/2}=2^{15}100 B^2 d^{11/2},\label{second-evaluation}
\end{equation}
for a constant $B$,
then we have
\[
    \left(\frac{5\sqrt{D}B(1+\log n)}{\sqrt{n}}\right)^2 \le \frac{\ve^2}8.
\]

Now consider a bounded and centered $\R^d$-valued random vector $X$ with $V=\E{XX^\top}$ nonsingular.
Then $B':=\sup\left\|V^{-1/2}X\right\|_2$ is finite.
Let $X_1, X_2, \ldots$ be independent copies of $X$.
Define $\R^D$-valued random vectors $Y_1, Y_2, \ldots$
by $Y_i:=(V^{-1/2}X_{(2i-1)d+1}, \ldots, V^{-1/2}X_{2id})^\top$ for each $i$.
Then, note that $\|Y_i\|_2\le \sqrt{2d}B'$.
By taking $B=\sqrt{2d}B'$ in \eqref{second-evaluation},
we have from Theorem \ref{zha18} that (for $\ve=2^{-13/2}d^{-7/4}$)
\begin{align*}
    \P{Z\in B_{d, \ve}}\le\frac18,\qquad \frac1{\ve^2}\mathcal{W}_2\left(\frac{Y_1+\cdots+Y_n}{\sqrt{n}}, Z\right) \le \frac18.
\end{align*}

From Proposition \ref{zhai},
we obtain
\[
    \P{\frac{Y_1+\cdots+Y_n}{\sqrt{n}} \in A_d}\ge \P{Z\in A_d} - \P{Z\in B_{d, \ve}}
    - \frac1{\ve^2}\mathcal{W}_2\left(\frac{Y_1+\cdots+Y_n}{\sqrt{n}}, Z\right)
    \ge \frac14.
\]
Therefore, $0$ is contained in the convex hull of $\{X_1,\ldots, X_{2dn}\}$
with probability at least $1/4$.
Since $(1-1/4)^3<1/2$, $N_X\le 6dn$ holds.
Therefore, our proof of Proposition \ref{prop-second} is complete.


\section{Extreme examples}\label{app-ex}
Before treating concrete examples,
we prove a proposition which is useful for evaluating $N_X$.
\begin{lem}\label{useful}
    For a random vector $X$ and its independent copies $X_1,X_2,\ldots$,
    define $\tilde{N}_X$
    as the minimum index $n$ satisfying $0\in\cv\{X_1, \ldots, X_n\}$.
    Then, we have
    \[
        \frac12\E{\tilde{N}_X}\le N_X \le 2\E{\tilde{N}_X}.
    \]
\end{lem}
\begin{proof}
    From the definition of $N_X$,
    $\P{0\in\{X_1, \ldots, X_{N_X-1}\}}<1/2$ holds.
    Thus $\P{\tilde{N}_X\ge N_X}\ge1/2$,
    and so we obtain
    $\E{\tilde{N}_X}\ge\frac12N_X$.
    
    For the other inequality,
    we use the evaluation $\P{\tilde{N}_X \ge kN_X} \le 2^{-k}$
    for each nonnegative integer $k$.
    As $\tilde{N}_X$ is a nonnegative discrete random variable,
    we have
    \[
        \E{\tilde{N}_X}
        =\sum_{n=1}^\infty\P{\tilde{N}_X \ge n}
        \le\sum_{k=0}^\infty N_X\P{\tilde{N}_X\ge k N_X}
        \le 2N_X.
    \]
\end{proof}

Note that all the examples given below satisfy $p_{d, X} = 0$.
They are given as one of the worst-case examples for uniform estimates of $N_X$
in Proposition \ref{naive} or Theorem \ref{thm-first}.
Let us start with the simplest extreme case.

\begin{eg}\label{eg-1dim}
    Let $d=1$.
    For an $\ve\in(0, 1)$,
    let $X$ be a random variable such that
    $\P{X=1/\ve} = \ve$ and $\P{X=-1/(1-\ve)} = 1-\ve$.
    Then $\E{X}=0$.
\end{eg}

In this example, we can explicitly calculate $p_{n, X}$ as
\[
    p_{n, X} = 1 - \ve^n - (1-\ve)^n.
\]
In particular, $p_{2, X} = 2\ve - 2\ve^2$.
We have $\lim_{\ve\searrow0}(1-\ve)^{1/2\ve}=e^{-1/2}=0.60\ldots$,
so $p_{\lceil1/2\ve\rceil, X} < 1/2$ holds for a sufficiently small $\ve$.
For such an $\ve$,
we have
\begin{equation}
    N_X \ge \frac1{2\ve} = \frac{1-\ve}2\frac{2}{p_{2, X}},
    \label{eq-1dim}
\end{equation}
and so $N_X\le \frac2{p_{2, X}}$ in Proposition \ref{naive} is sharp up to constant.

For $\ve\in(0, 1/2)$,
$N_X$ can also be evaluated above as
$N_X \le 2\E{\tilde{N}_X}
   \le 2\left(\frac1\ve + \frac1{(1-\ve)}\right)$
by using Proposition \ref{useful}.
We also have $\alpha_X = \ve$ for $\ve\in(0, 1/2)$,
so
\[
    \inf_{\text{$X$:$1$-dimensional}}\alpha_XN_X \le 2 + \frac{2\ve}{1-\ve} \to 2 \quad (\ve\to 0).
\]

As the variance is $V=\E{X^2}=\frac1\ve + \frac1{1-\ve} = \frac1{\ve(1-\ve)}$,
we have
\begin{align*}
    \E{\left\lvert V^{-1/2}X\right\rvert^3}^2
    &=V^{-3}\left(\frac1{\ve^2} + \frac1{(1-\ve)^2}\right)^2\\
    &=\ve^3(1-\ve)^3\left(\frac1{\ve^4} + 
    \frac2{\ve^2(1-\ve)^2} + \frac1{(1-\ve)^4}\right)\\
    &=\frac1\ve + \ord{1}.
\end{align*}
Therefore, from \eqref{eq-1dim}, we obtain
\[
    \sup\left\{
        \E{\left\lvert V^{-1/2}X\right\rvert^3}^{-2}N_X
        \lmid
        \begin{array}{c}
            \text{$X$ is $1$-dimensional},\ \E{X}=0,\\
            V=\E{X^2}\in(0, \infty),\ \E{\left\lvert V^{-1/2}X\right\rvert^3}<\infty
        \end{array}
    \right\}
    \ge \frac12,
\]
which is what is mentioned in Remark \ref{optimality} when $d=1$.

The next example is a multi-dimensional version of the previous one.

\begin{eg}\label{eg-multi}
    Let $d\ge2$.
    Let $\{e_1,\ldots,e_d\}\subset\R^d$ be the standard basis of $\R^d$.
    Let us first consider,
    for an arbitrary $\ve\in (0, 1)$,
    a random vector $X$ given by
    \[
        X = Y\left(\sum_{i=1}^{d-1} Z^ie_i - \frac1{1-\ve}e_d\right) + \frac1\ve(1-Y)e_d,
    \]
    where $\P{Y=1}=1-\ve$, $\P{Y=0}=\ve$ and
    $Z^1,\ldots,Z^{d-1}$ are independent uniform random variables over $[-1, 1]$.
    (also independent from $Y$).
    Namely, $X$ is $\ve^{-1} e_d$ with probability $\ve$
    and a $(d-1)$-dimensional uniform vector over a box
    on the hyperplane $\{x\in\R^d\mid e_d^\top x = - (1-\ve)^{-1}\}$
    otherwise.
    $\E{X}=0$ also holds.
\end{eg}

Let us estimate $p_{d+1,X}, p_{2d, X}$ and $N_X$ for this $X$.
To contain the origin in the convex hull,
we have to observe at least one $X_i$ with $Y=0$.
Therefore, for an $\ve \ll 1/d$, we have
\begin{align*}
    p_{d+1, X}
    &=(d+1)\ve(1-\ve)^d2^{-(d-1)} = \frac{d+1}{2^{d-1}}\ve\left(1+\ord{d^2\ve^2}\right)\\
    p_{2d, X}
    &=\sum_{k=1}^d \binom{2d}{k}\ve^k(1-\ve)^{2d-k} p_{2d-k, X'} \\
    &= 2d\ve p_{2d-1, X'} + \ord{d^2\ve^2} = d\left(1 +  \frac1{2^{2d-2}}\binom{2d-2}{d-1}\right)\ve + \ord{d^2\ve^2}\\ 
    &\ge d\left(1+\frac1{2\sqrt{d-1}}\right)\ve + \ord{d^2\ve^2},
\end{align*}
where $X'$ represents a $(d-1)$-dimensional uniform random vector over the box $[-1, 1]^{d-1}$.
We can see that $p_{2d, X}\gtrsim 2^{d-1}p_{d+1, X}$ holds for a small $\ve$
as Remark \ref{rem-important} suggests.

For the calculation of $N_X$,
we can exploit Proposition \ref{useful}.
We first bound the expectation of $\tilde{N}_X$.
For independent copies $X_1, X_2, \ldots$ of $X$,
let $N_1$ be the minimum integer $n$ satisfying $X_n = \ve^{-1}e_d$.
We also define $N_2$ as the minimum integer $n$ satisfying
$-(1-\ve)^{-1}e_d\in\cv\{X_1,\ldots,X_n\}$.
Then, $\tilde{N}_X = \max\{N_1, N_2\}$ holds.
Thus we have $N_1 \le \tilde{N}_X \le N_1 + N_2$.
$\E{N_1}=1/\ve$ clearly holds.
For $N_2$,
we can evaluate (again using $X'$) as
\[
    \E{N_2}=\frac1{1-\ve}\E{\tilde{N}_{X'}} \le \frac{2N_{X'}}{1-\ve}
    =\frac{4(d-1)}{1-\ve},
\]
where we have used Proposition \ref{useful} for the inequality.
Therefore, from Proposition \ref{useful}, we obtain
\begin{equation}
    \frac1{2\ve} \le \frac12\E{\tilde{N}_X} \le N_X \le 2\E{\tilde{N}_X} \le \frac2\ve + \frac{8(d-1)}{1-\ve}.
    \label{eg1-eval}
\end{equation}

We finally compare the naive general estimate $N_X\le \frac{n}{p_{n, X}}$ in Proposition \ref{naive}
with this example.
From \eqref{eg1-eval},
we have
\[
    \frac{N_Xp_{2d, X}}{2d} \ge
    \frac{p_{2d, X}}{4d\ve}
    \ge \frac14 + \frac1{8\sqrt{d-1}} + \ord{d\ve}.
\]
Therefore, the evaluation $N_X\le \frac{2d}{p_{2d, X}}$ is sharp even for small $p_{2d, X}$
up to constant in the sense that
\[
    \lim_{\ve\to 0}\sup_{\substack{\text{$X$:$d$-dimensional}\\ p_{2d, X}<\ve}}\frac{N_Xp_{2d, X}}{2d} \ge \frac14 + \frac1{8\sqrt{d-1}}
\]
holds.

Also in this example, we have $\alpha_X = \ve$ for $\ve\in(0, 1/3)$.
Hence, combined with \eqref{eg1-eval},
we have
\[
    \alpha_XN_X
    \le \ve\left(\frac2\ve + \frac{8(d-1)}{1-\ve}\right)
    =2 + \frac{8(d-1)\ve}{1-\ve} \to 2 \quad (\ve\to 0).
\]
Therefore, we have $\inf_{\text{$X$:$d$-dim}}\alpha_XN_X\le 2$.

We next evaluate the value of $\E{\left\|V^{-1/2}X\right\|_2^3}$,
where $V=(V^{ij})$ is the covariance matrix of $X$ with respect to the basis $\{e_1,\ldots,e_d\}$.
Then, for $(i,j)\in\{1,\ldots,d-1\}^2$, we obtain
\begin{align*}
    V^{ij} &= \E{Y^2Z^iZ^j} = \E{Y^2}\E{Z^iZ^j} =\frac{1-\ve}2\delta^{ij},
    \tag{$\delta^{ij}$: Kronecker's delta}\\
    V^{id} &= \E{YZ^i\left( - \frac{Y}{1-\ve} + \frac{1-Y}\ve\right)}
    = \E{Z^i}\E{Y\left( - \frac{Y}{1-\ve} + \frac{1-Y}\ve\right)} = 0
\end{align*}
by using the independence of $Y$, $Z_1,\ldots, Z_{d-1}$.
For the $V^{dd}$, we have
\[
    V^{dd}=\frac1{1-\ve} + \frac1\ve = \frac1{\ve(1-\ve)}.
\]
Therefore,
$V^{-1/2}X$
can be explicitly written as
\[
    V^{-1/2}X = Y\left( \sqrt{\frac{2}{1-\ve}}\sum_{i=1}^{d-1}Z^ie_i
    - \sqrt{\frac{\ve}{1-\ve}}e_d\right)
    + \sqrt{\frac{1-\ve}{\ve}}(1-Y)e_d.
\]
Thus we have
\[
    \left\|V^{-1/2}X\right\|_2^2\le Y\frac{2(d-1) + \ve}{1-\ve} + (1-Y)\frac{1-\ve}{\ve},
\]
and so
\[
    \E{\left\|V^{-1/2}X\right\|_2^3}\le \frac{(2(d-1)+\ve)^{3/2}}{\sqrt{1-\ve}}
    +\frac{(1-\ve)^{3/2}}{\sqrt{\ve}}
    \le 4d^{3/2} + \ve^{-1/2}
\]
holds when $0<\ve<1/2$.
By using \eqref{eg1-eval},
we obtain
\[
    \frac{N_X}{\E{\left\|V^{-1/2}X\right\|_2^3}^2} \ge \frac1{2\ve (4d^{3/2} + \ve^{-1/2})^2}
    =\frac1{2(4{d^{3/2}\ve^{1/2}} + 1)^2}.
\]
Therefore, by taking $\ve\to 0$, we finally obtain the estimate
\[
    \sup\left\{\frac{N_X}{\E{\left\|V^{-1/2}X\right\|_2^3}^2}
        \lmid
        \begin{array}{c}
            \text{$X$ is $d$-dimensional},\ \E{X}=0,\\
            \text{$V=\E{X^2}$ is nonsingular},\ \E{\left\| V^{-1/2}X\right\|^3}<\infty
        \end{array}
    \right\}
    \ge \frac12
\]
as mentioned in Remark \ref{optimality}.

\end{document}